\documentclass[10pt]{article}
\usepackage{amsmath}
\usepackage{graphics}
\usepackage{latexsym}
\usepackage{amsfonts}
\usepackage{mathrsfs}
\usepackage{amssymb}
\usepackage{amsthm}
\usepackage{bbm}
\usepackage{enumerate}
\usepackage{accents}
\usepackage{statex}
\usepackage{xcolor}
\usepackage{hyperref}
\usepackage{geometry}
\numberwithin{equation}{section}
\newtheorem{proposition}{Proposition}[section]
\newtheorem{definition}{Definition}[section]
\newtheorem{theorem}{Theorem}[section]
\newtheorem{corollary}{Corollary}[section]
\newtheorem{lemma}{Lemma}[section]
\newtheorem{remark}{Remark}[section]

\newcommand{\ba}{\begin{array}}
\newcommand{\ea}{\end{array}}
\newcommand{\bal}{\begin{aligned}}
\newcommand{\eal}{\end{aligned}}
\newcommand{\bals}{\begin{aligned*}}
\newcommand{\eals}{\end{aligned*}}
\newcommand{\be}{\begin{equation}}
\newcommand{\ee}{\end{equation}}
\newcommand{\bes}{\begin{equation*}}
\newcommand{\ees}{\end{equation*}}
\newcommand{\bea}{\begin{eqnarray}}
\newcommand{\eea}{\end{eqnarray}}
\newcommand{\beas}{\begin{eqnarray*}}
\newcommand{\eeas}{\end{eqnarray*}}

\newcommand{\abs}[1]{\left| #1\right|}

\def\no{\noindent}

\def\ss{\smallskip}
\def\ms{\medskip}

\def \d {\delta}
\def \R {{\bf R}}

\def \esp {{[0,T]\times \R^m}}
\def \P {{\cal P}}
\def \m {{\cal M}}
\def \cf {\mathcal{F}}
\def \bp {{\bf{P}}}
\def \tx {(t,x)\in \esp}
\def \xs {X_s^{t,x}}
\def \xu {X_u^{t,x}}
\def \xso {X_s^{0,x}}
\def \xto {X_t^{0,x}}

\def \xro {X_r^{0,x}}
\def \xTo {X_T^{0,x}}
\def \xT {X_T^{t,x}}
\def \ytuv {Y_t^{1,(u,v)}}
\def \ytuvs {Y_t^{1,(u, v^*)}}

\def \yts {\tilde{Y}_s^{1n}}
\def \zts {\tilde{Z}_s^{1n}}

\def \ybtn {\bar{Y}_t^{1n}}
\def \ybsn {\bar{Y}_s^{1n}}
\def \ybsnt {\bar{Y}_s^{2n}}
\def \ybtnt {\bar{Y}_t^{2n}}
\def \zbtn {\bar{Z}_t^{1n}}
\def \zbsn {\bar{Z}_s^{1n}}
\def \zbtnt {\bar{Z}_t^{2n}}
\def \zbsnt {\bar{Z}_s^{2n}}
\def \ybtm {\bar{Y}_t^{1m}}
\def \ybsm {\bar{Y}_s^{1m}}

\def \zbsm {\bar{Z}_s^{1m}}

\def \ztuvs {Z_t^{1,(u, v^*)}}
\def \zsuvs {Z_s^{1,(u, v^*)}}
\def \ztuv {Z_t^{1,(u,v)}}

\def \ybuv {\bar{Y}_t^{1,(u,v)}}
\def \ybt {\bar{Y}_t^1}
\def \ybs {\bar{Y}_s^1}
\def \ybtt {\bar{Y}_t^2}
\def \ybst {\bar{Y}_s^2}
\def \zbt {\bar{Z}_t^1}
\def \zbs {\bar{Z}_s^1}
\def \zbtt {\bar{Z}_t^2}
\def \zbst {\bar{Z}_s^2}
\def \zbtuv {\bar{Z}_t^{1,(u,v)}}
\def \zbsuv {\bar{Z}_s^{1,(u,v)}}
\def \dtuvs {D_t^{u, v^*}}
\def \dsuvs {D_s^{u, v^*}}
\def \dtusvs {D_t^{u^*,v^*}}
\def \dsusvs {D_s^{u^*,v^*}}
\def \bsuv {B_s^{u,v}}

\def \euv {\textbf{E}^{u,v}}
\def \us {u^*}
\def \vs {v^*}
\def \ito {It\^{o}}
\begin{document}
\date{\today}
\title{\vspace{-.3em}
Risk-sensitive Nonzero-sum Stochastic Differential Game with Unbounded Coefficients\\
\author{
  Said Hamad\`ene
 \thanks{
 Universit\'e du Maine, LMM, Avenue Olivier Messiaen, 72085 Le Mans, Cedex 9, France. hamadene@univ-lemans.fr.
 }
  \  and Rui Mu%
  \thanks{
  Universit\'e du Maine, LMM, Avenue Olivier Messiaen, 72085 Le Mans, Cedex 9, France; School of Mathematics, Shandong University, Jinan 250100, China. rui.mu.sdu@gmail.com.
  }
  \thanks{Supported in part by the Natural Science Foundation for Young Scientists of Jiangsu Province, P.R. China (No. BK20140299).}
 }
}
\maketitle
 \vspace{-1em}

\begin{abstract}
This article is related to risk-sensitive nonzero-sum stochastic differential games in the Markovian framework. This game takes into account the attitudes of the players toward risk and the utility is of exponential form. We show the existence of a Nash equilibrium point for the game when the drift is no longer bounded and only satisfies a linear growth condition. The main tool is the notion of backward stochastic differential equation, which in our case, is multidimensional with continuous generator involving both a  quadratic term of $Z$ and a stochastic linear growth component with respect to $Z$.
\end{abstract}

\medskip

\no {$\bf Keywords$}: risk-sensitive; nonzero-sum stochastic differential games; Nash equilibrium point; backward stochastic differential equations.
\medskip

\no {\bf AMS subject classification:} 49N70; 49N90; 91A15.

%%%%%%%%%%%%%%%%%%%%%%%%%%%%%%%%%%%%%%%%%%%%%%%%%%%%%%%%%%%%%%%%%%%%%%%%%%%%%%%%%
%%%%%%%%%%%%%%%%%%%%%%%%%%%%%%%%%%%%%%%%%%%%%%%%%%%%%%%%%%%%%%%%%%%%%%%%%%%%%%%%%
%%%%%%%%%%%%%%%%%%%%%%%%%%%%%%%%%%%%%%%%%%%%%%%%%%%%%%%%%%%%%%%%%%%%%%%%%%%%%%%%%Introduction

\section{Introduction}
We consider, in this article, a risk-sensitive nonzero-sum stochastic differential game model. Assume that we have a system which is controlled by two players. Each one imposes an admissible control which is an adapted stochastic process denoted by $u=(u_t)_{t\leq T}$ (resp. $v=(v_t)_{t\leq T}$) for player 1 (resp. player 2). The state of the system is described by a process $(x_t)_{t\leq T}$ which is the solution of the following stochastic differential equation: 
\begin{equation}\label{diff}
dx_t=f(t,x_t,u_t,v_t)dt+\sigma(t,x_t)dB_t \text{ for } t\leq T \text{ and } x_0=x,
\end{equation}
where $B$ is a Brownian motion. We establish this game model in a two-player framework for an intuitive comprehension.  All results in this article are applicable to the multiple players case.  Naturally, the control action is not free and has some risks. A \textit{risk-sensitive nonzero-sum stochastic differential game} is a game model which takes into account the attitudes of the players toward risk.  More precisely speaking, for player $i=1,2,$  the utility (cost or payoff) is given by  the following exponential form
$$
J^i(u,v)= \textbf{E}[e^{\theta \{\int_0^T h_i(s, x_s,u_s,v_s)ds+g^i(x_T)\}}].
$$
The parameter $\theta$ represents the attitude of the player with respect to  risk. What we are concerned here is a nonzero-sum stochastic differential game which means that the two players are of cooperate relationship. Both of them would like to minimize the cost and no one can cut more by unilaterally changing his own control. Therefore, the objective of the game problem is to find a \textit{Nash equilibrium point} $(u^*, v^*)$ such that, 
$$
J^1(u^*, v^*)\leq J^2(u, v^*)\text{ and } J^2(u^*, v^*)\leq J^2(u^*, v),
$$
for any admissible control $(u,v)$.

Let us illustrate now, why $\theta$, in the cost function, can reflect the risk attitude of the controller. From the economic point of view, we denote by 
$
G^i_{u,v}=\int_0^T h_i(s, x_s, u_s, v_s)ds+g^i(x_T)
$
the wealth of each controller and for a smooth function $F(z)$, let 
$
F(G^i_{u,v})
$
be the cost might be brought from the wealth. The two participates would like to minimize the expected cost $\textbf{E}[F(G^i_{u,v})]$. A notion of \textit{risk sensitivity} is proposed
 %by economists Kenneth J. Arrow (see \cite{arrow}) and John W. Pratt (see \cite{pratt}) 
 as follows:
$$
\gamma=\frac{F^{''}(G^i)}{F^{'}(G^i)}.
$$
It is a reasonable function to reflect the trend, more precise, the curvature of cost $F$ with respect to the wealth $G^i$. See W.H. Fleming's work \cite{fleming} for more details. In the present paper, we choose the utility function $F(z)$ as an exponential form $F(z)=e^{\theta z}$. Both theoretical and practical experiences have shown that it is often appropriate to use an  exponential form of utility function. Therefore, the risk sensitivity $\gamma$ is exactly the parameter $\theta$.

We explain this specific case $\gamma=\theta$ in the following way. The expected utility $J^i(u,v)=\textbf{E}[e^{\theta G^i_{u,v}}]$ is certainty equivalent to 
$$
\varrho^i_{\theta}(u,v):=\theta^{-1}\ln \textbf{E}[e^{\theta G^i_{u,v}}].
$$
By certainty equivalent, we refer to the minimum premium we are willing to pay to insure us against some risk (or the maximum amount of money we are willing to pay for some gamble). Then, 
$
\varrho^i_{\theta}(u,v)\sim \textbf{E}[G^i_{u,v}]+\frac{\theta}{2}\text{Var}(G^i_{u,v})
$
provided that $\theta\text{Var}(G^i_{u,v})$ is small (Var$(.)$ is the variance ). Hence, minimizing $J^i(u,v)$ is equivalent to minimize $\varrho_\theta^i(u,v)$. The variance Var$(G^i_{u,v})$ of the wealth reflects the risk of  decision to a certain extent.  Therefore, it is obvious that when $\theta>0$, the less risk the better. Such a decision maker in economic markets will have a \textit{risk-averse} attitude. On the contrary, when $\theta<0$, the optimizer is called \textit{risk-seeking}. Finally, if $\theta=0$, this situation corresponds to the risk-neutral controller. Without loss of generality, we set $\theta=1$ in this work.

About the risk-sensitive stochastic differential game problem, including  nonzero-sum, zero-sum and mean-field cases, there are some previous works. Readers are referred to \cite{Basar2, H2003, Fleming1, Fleming2, James, Basar1} for further acquaintance. Among those results, a particular popular approach is partial differential equation, such as \cite{Basar2, Fleming1, Fleming2, James, Basar1} with various objectives. Another method is through backward stochastic differential equation (BSDE) theory, see \cite{H2003}. The nonlinear BSDE is introduced by Pardoux and Peng \cite{peng1990} and developed rapidly in the past two decades. The notion of BSDE is proved as an efficient tool to deal with stochastic differential game. It has been used in the risk neutral case, see \cite{hamadene1997, H1998}. About Other applications such as in the field of mathematic finance, we refer the work by El-Kaoui et al. \cite{karoui} (1997). A complete review on BSDEs theory as well as some new results on nonlinear expectation are introduced in a survey paper by Peng (2010) \cite{peng2011}.

In the present paper, we study the risk-sensitive nonzero-sum stochastic differential game problem through BSDE in the same line as article by  El-Karoui and  Hamad\`ene \cite{H2003}. However in \cite{H2003}, the setting of game problem concerns only the case when the drift coefficient $f$ in diffusion \eqref{diff} is bounded. This constrain is too strict to some extent.    Therefore, our motivation is to relax as much as possible the boundedness of the coefficient $f$.  We assume that $f$ is not bounded any more but instead, has a linear growth condition. It is the main novelty of this work.  To our knowledge, this general case has not been studied in the literature. Finally, we show the existence of  Nash equilibrium point for this game.  We provide a link between the game which we constructed and BSDE. The existence of the NEP is equivalent to the existence of solutions for a related BSDE, which is multiple-dimensional with continuous generator involving both stochastic linear growth and quadratic terms of $z$. The difference with \cite{H2003} is that the linear term of $z$ is of linear growth $\omega$ by $\omega$ due to the linear growth of $f$. Under the generalized Isaacs hypothesis and domination property of solutions for \eqref{diff}, which holds when the uniform elliptic condition on $\sigma$ is satisfied, we show that the associated BSDE has a solution which then provides the NEP for our game.

The paper is organized as follows:

In Section 2, we present the precise model of risk-sensitive nonzero-sum stochastic differential game and necessary  hypotheses on related coefficients. In Section 3, we firstly state some useful lemmas. Particularly, Lemma \ref{density function lp bounded} and Corollary \ref{coro lp}, which corresponding to the  integrability of Dol\'ean-Dade exponential local martingale, play a crucial role.  Then, the link beween game and BSDE is demonstrated by Proposition \ref{prop BSDE uv}. The utility function is characterized by the initial value of a BSDE. Then, it turns out by Theorem \ref{th nash} that the existence of the NEP for this game problem is equivalent to the existence of some specific BSDE which is multiple dimensional, with continuous generator involving a quadratic term and a linear growth term of $Z$, $\omega$ by $\omega$.  Finally, we show, in Section 4, the solutions for this specific BSDE exist when the generalized Isaacs condition is fulfilled and the law of the dynamic of the system satisfies the $L^q$-domination condition. The latter condition is naturally holds if the diffusion coefficient $\sigma$ satisfies the well-known uniform elliptic condition. 
Our method to deal with this BSDE with non-regular quadratic generator is that we firstly cancel the quadratic term by applying the exponential transform, then, we take an approximation of the new generator. Besides, in Markovian framework, those approximate processes can be expressed via some deterministic functions. We then provide uniform estimates of the processes, as well as the growth properties of the corresponding deterministic functions. Later, the convergence result is proved. At the end, by taking the inverse transform, the proof for the existence is finished.

%%%%%%%%%%%%%%%%%%%%%%%%%%%%%%%%%%%%%%%%%%%%%%%%%%%%%%%%%%%%%%%%%%%%%%%%%%%%%%%%%
%%%%%%%%%%%%%%%%%%%%%%%%%%%%%%%%%%%%%%%%%%%%%%%%%%%%%%%%%%%%%%%%%%%%%%%%%%%%%%%%%
%%%%%%%%%%%%%%%%%%%%%%%%%%%%%%%%%%%%%%%%%%%%%%%%%%%%%%%%%%%%%%%%%%%%%%%%%%%%%%%%%Statement of problem
\section{Statement of the risk-sensitive game}

In this section, we will give some basic notations, the preliminary
assumptions throughout this paper, as well as the statement of the
risk-sensitive nonzero-sum stochastic differential game.  Let
$(\Omega, \mathcal{F}, \bf{P})$ be a probability space on which we
define a $d$-dimensional Brownian motion $B=(B_t)_{0\leq t\leq T}$ with integer $d\geq 1$ and fixed $T>0$. Let us denote by ${\bf{F}}=\{\mathcal{F}_t, 0\leq t\leq T\}$, the natural filtration generated by process $B$ and augmented by
$\mathcal{N}_{\bf{P}}$ the $\bf{P}$-null sets, \emph{i.e.}
$\mathcal{F}_{t}=\sigma\{B_{s},$\ $s\leq t\}\vee\mathcal{N}_{\bf{P}}%
$.
\medskip

Let $\mathcal{P}$ be the $\sigma$-algebra on $[0,T]\times \Omega$ of
$\mathcal{F}_t$-progressively measurable sets. Let $p\in [1,\infty)$ be real constant and $t\in [0,T]$ be fixed. We then define the following spaces:

\begin{itemize}
 \item $\mathcal{L}^p$ = $\{\xi: \mathcal{F}_t$-measurable and$\ \R^m$-valued random variable such that$\
 \textbf{E}[|\xi|^p]<\infty\}$;
 \item  $\mathcal{S}_{t,T}^p(\R^m)$ = $\{\varphi=(\varphi_s)_{t\leq s \leq T}$: $\mathcal{P}$-measurable, continuous and $\R^m$-valued such that $\textbf{E}[\sup\limits_{s\in [t,T]}|\varphi_s|^p]< \infty \}$; 
 \item $\mathcal{H}_{t,T}^p(\R^m)=\{\varphi=(\varphi_s)_{t \leq s\leq T}: \mathcal{P}$-measurable and $\R^m$-valued such that $\ \textbf{E}[(\int_{t}^T|\varphi_s|^2ds)^{\frac{p}{2}}]< \infty\}$;
 \item  $\mathcal{D}_{t,T}^p(\R^m)=\{\varphi=(\varphi_s)_{t\leq s\leq T}: \mathcal{P}$-measurable and $\R^m$-valued such that $\  \textbf{E}[\sup\limits_{s\in [t, T]}e^{p\varphi_s}]< \infty \}$.
% \item $\mathcal{M}_{T}(\R^m)=\{\varphi=(\varphi_t)_{0\leq t\leq T}: \mathcal{P}$ -measurable and $\R^m$-valued such that $\int_0^T|\varphi_s|^2ds< \infty$ {\bf{P}}-a.s.\}.\medskip
\end{itemize}%
\no Hereafter, $\mathcal{S}_{0,T}^p(\R^m)$, $\mathcal{H}_{0,T}^p(\R^m)$, $\mathcal{D}_{0,T}^p(\R^m)$ are simply denoted by $\mathcal{S}_T^p(\R^m)$, $\mathcal{H}_{T}^p(\R^m)$, $\mathcal{D}_{T}^p(\R^m)$. The following assumptions are in force throughout this paper. Let
$\sigma$ be the function defined as:
$$\begin{array}{cl}\sigma: [0,T]\times \R^m&\longrightarrow \R^{m\times
m}\\ (t,x)&\longmapsto \sigma (t,x)\end{array}$$ which satisfies the
following assumptions:
\begin{flushleft}
    \textbf{Assumptions (A1)}
\end{flushleft}
\begin{description}
  \item[(i)] $\sigma$ is uniformly Lipschitz w.r.t $x$. \emph{i.e.} there exists a constant $C_1$  such that,\\
  $$\forall t \in [0, T], \forall \ x, x^{\prime} \in \R^m,\quad
\abs{ \sigma(t,x)-\sigma(t, x^{\prime})} \leq C_1 \abs{x-
x^{\prime}}.$$
%  \item[(ii)] $\sigma$ is bounded \emph{i.e.} there exits a constant $C_{\sigma}$ such that $$\forall (t,x)\in
%  \esp, \quad \abs {\sigma(t, x)} \leq C_{\sigma}.$$
%  \item[(iii)] $\sigma$ satisfies the non-degeneracy condition.
%  \emph{i.e.} there exists a $\alpha>0$ such that $$\forall (t,x)\in
%  \esp,\quad
%  \alpha^{-1}I \geq \sigma(t,x)\sigma^\top(t,x) \geq \alpha I,$$
%  \noindent where $I$ is the identity matrix.
%  \item[(iv)] xxxx can be deleted?????? $\sigma$ is invertible and its inverse $\sigma^{-1}$
%  satisfies $$\forall (t,x)\in
%  \esp, \quad \abs{\sigma^{-1}(t,x)}\leq C_2. $$
\item[(ii)] $\sigma$ is invertible and bounded and its inverse is bounded, \textit{i.e.}, there exits a constant $C_{\sigma}$ such that $$\forall (t,x)\in
  \esp, \quad \abs {\sigma(t, x)} +\abs{\sigma^{-1}(t,x)}\leq C_{\sigma}.$$
  \end{description}
\begin{remark} {\bf Uniform elliptic condition.}\\
\noindent Under Assumptions (A1), we can verify that, there exists a real constant $\epsilon>0$ such that for any $(t,x)\in \esp $,
\begin{equation}\label{horm}
\epsilon.I\leq \sigma(t,x).\sigma^\top(t,x)\leq \epsilon^{-1}.I
\end{equation}
where $I$ is the identity matrix of dimension $m$. 
\end{remark}

We consider, in this article the 2-player case. The general multiple players game is a straightforward adaption. 

For $\tx$, let $X= (\xs)_{s\leq T}$ be the solution of the following stochastic differential equation:
\be\label{SDE sigma} 
\left\{ \bal
\xs &= x+ \int_t^s \sigma(u,\xu)dB_u%
,\ s\in\left[t,T\right]  ;\ss\\
\xs &= x,\ s\in [0,t]. \eal \right.
\ee

\no Under Assumptions (A1) above, we know such $X$ exists and is unique (see Karatzas and Shreve,
pp.289, 1991\cite{KS}). Let us now denote by $U_1$ and $U_2$ two compact metric spaces and
let $\m _1$ (resp. $\m _2$) be the set of $\P$-measurable processes
$u=(u_t)_{t\leq T}$ (resp. $v=(v_t)_{t\leq T}$) with values in $U_1$
(resp. $U_2$).  We denote by $\m$ the set $\m_1\times \m_2$, hereafter $\mathcal{M}$ is called the set of admissible controls.  We then introduce two Borelian functions \bes \bal
&f:\esp\times U_1\times U_2\longrightarrow\R^m,\\
&h_i\  (\text{resp}.\  g^i):\esp\times U_1\times U_2\
(\text{resp}.\  \R^m)\longrightarrow \R, \ i=1,2,%
\eal
\ees
which satisfy:%
\begin{flushleft}
\textbf{Assumptions (A2)}
\end{flushleft}

\begin{description}
  \item[(i)] for any $(t,x)\in [0,T]\times \R^m$, $(u,v)\mapsto f(t,x,u,v)$ is continuous on $U_1\times U_2$. Moreover $f$ is of linear growth w.r.t $x$, \emph{i.e.} there exists a constant $C_f$
  such that $\abs{f(t,x,u,v)}\leq C_f(1+\abs{x}),  \forall
  (t,x,u,v)\in [0,T]\times \R^m\times U_1\times U_2.$
  \item[(ii)] for any $(t,x)\in [0,T]\times \R^m$, $(u,v)\mapsto h_i(t,x,u,v)$ is continuous on $U_1\times U_2$, $i=1,2$.
   Moreover, for $i=1,2$, $h_i$ is of sub-quadratic growth w.r.t $x$, 
  \emph{i.e.}, there exist constants $C_h$ and $1<\gamma<2$ such that $\abs{h_i(t,x,u,v)}\leq C_h(1+\abs{x}^{\gamma}),\ \forall
  (t,x,u,v)\in [0,T]\times \R^m\times U_1\times U_2.$
  \item[(iii)] the functions $g^i$ are of sub-quadratic growth with respect to $x$, \emph{i.e.} %
  there exist constants $C_g$ and $1<\gamma<2$ such that $\abs{g^i(x)}\leq C_g(1+\abs{x}^{\gamma}),
\forall x\in \R^m$, for i=1, 2.
  \end{description}
\medskip

For $(u,v)\in \m$, let $\bp^{u,v}_{t,x}$ be the measure on $(\Omega,
\mathcal{F})$ defined as follows:

\be \label{new prob} d{\bp}^{u,v}_{t,x} = \zeta_T\big(\int_0^{.}\sigma^{-1}(s, \xs)f(s,
\xs, u_s, v_s)dB_s\big)d\bp, \ee

\noindent where for any $(\cf_t, \bp)$-continuous local martingale
$M= (M_t)_{t\leq T}$, \be\label{density fun} \zeta(M):=\big(\exp\{M_t-\frac{1}{2}\langle M\rangle_t\}\big)_{t\leq T},\ee 

\no where {\small{$\langle\ \rangle_{.}$}} denotes the quadratic variation process.  We could deduce from Assumptions (A1), (A2)-(i) on $\sigma$
and $f$ that $\bp^{u,v}_{t,x}$ is a probability on $(\Omega, \cf)$ (see Appendix A, \cite{H2003} or \cite{KS} pp.200). By
Girsanov's theorem (Girsanov, 1960 \cite{Gir}, pp.285-301), the
process $B^{u,v}:= (B_s-\int_0^s \sigma^{-1}(r,X^{t,x}_r)f(r,
X^{t,x}_r, u_r, v_r) dr)_{s\leq T}$ is a $(\cf_s,
\bp^{u,v}_{t,x})$-Brownian motion and $(\xs)_{s\leq T}$ satisfies the following stochastic differential equation:

\begin{equation}\label{SDEf}
\left\{ \bal d\xs &=
f(s,\xs,u_s,v_s)ds + \sigma(s,\xs)d\bsuv%
,\ s\in\left[t,T\right]  ;\ss\\
\xs &= x,\ s\in [0,t]. \eal \right.
\end{equation}

\no As a matter of fact, the process $(\xs)_{s\leq T}$ is
not adapted with respect to the filtration generated by the Brownian
motion $(\bsuv)_{s\leq T}$ any more, therefore $(\xs)_{s\leq T}$ is called a weak solution for the SDE (\ref{SDEf}). Now the system is controlled by player 1 (resp. Player 2) with $u$ (resp. $v$).
\medskip

Now, let us fix $(t,x)$ to $(0,x_0)$, i.e., $(t,x)=(0,x_0)$. For a general risk preference coefficient $\theta$, we define the \textit{costs} (or \textit{payoffs}) of the players for $(u,v)\in \mathcal{M}$ by:
\be \label{payoff}
J^i(u,v)= \textbf{E}^{u,v}_{0,x_0}\Big[e^{\theta \{\int_0^T h_i(s, X_s^{0,x_0},u_s,v_s)ds+g^i(X_T^{0,x_0})\}}\Big],\  i=1,2
\ee
\no where $\euv_{0,x_0}(.)$ is the expectation under the probability $\bp^{u,v}_{0,x_0}$. Hereafter $\textbf{E}_{0,x_0}^{u,v}$(resp. $\bp_{0,x_0}^{u,v}$) will be simply denoted by $\textbf{E}^{u,v}$(resp. $\bp^{u,v}$). The functions $h_1$ and $g^1$ (resp. $h_2$ and $g^2$) are, respectively, the \textit{instantaneous} and \textit{terminal costs} for player 1 (resp. player 2). The player is called risk-averse (resp. risk-seeking) if $\theta>0$ (resp. $\theta<0$). Since the resolution of the problem is the same in all cases ($\theta>0$, $\theta<0$ or $\theta=0$), without loss of generality, we assume $\theta=1$ in (\ref{payoff}) for simplicity below.
\medskip
%%%%%%%%%%%%%%%%%%%%%%%%%%%%%%%%%%%%%%%%aim of the problem

In this article, the quantity $J^i(u,v)$ is the cost that player $i\ (i=1,2)$ has to pay for his control on the system. The problem is to find a pair of admissible controls $(\us, \vs)$ such that:
\bes
J^1(\us,\vs)\leq J^1(u, \vs)  \ \text{and}\ J^2(\us,\vs)\leq J^2(\us,v), \ \forall(u,v)\in \cal{M}.
\ees
\no The control $(\us, \vs)$ is called a \textit{Nash equilibrium point} for the risk-sensitive nonzero-sum stochastic differential game which means that each player chooses his best control, while, an equilibrium is a pair of controls, such that, when applied, no player will lower his/her cost by unilaterally changing his/her own control.
\medskip

Let us introduce now the \textit{Hamiltonian functions} for this game, for $i=1,2,$  by $H_i:
[0, T]\times \R^{2m}\times U_1\times U_2\rightarrow \R$, associate:
\begin{equation}\label{definition of hamiltonian function}
H_i(t, x, p, u, v)= p\sigma^{-1}(t,x)f(t, x, u, v)+ h_i(t, x, u, v).
\end{equation}
Besides,  we introduce the following assumptions which will play an important role in the proof of existence of equilibrium point.
\medskip

\noindent\textbf{Assumptions (A3)}\\
\\
\textbf{(i)} \textbf{Generalized Isaacs condition:} There exist two borelian applications $u_1^*$, $u_2^*$ defined
  on $[0,T]\times \R^{3m}$, with values in $U_1$ and $U_2$ respectively, 
  such that for any $(t, x, p, q, u, v)\in [0, T]\times \R^{3m}\times
  U_1\times U_2$, we have:
  $$H_1^*(t,x,p,q)= H_1(t, x, p, u_1^*(t,x,p,q), u_2^*(t,x,p,q))\leq
  H_1(t,x,p,u,u_2^*(t,x,p,q))$$
and
  $$H_2^*(t,x,p,q)= H_2(t, x, q, u_1^*(t,x,p,q), u_2^*(t,x,p,q))\leq
  H_2(t,x,q,u_1^*(t,x,p,q), v).$$
  \medskip
\noindent \textbf{(ii)} The mapping $(p,q)\in \R^{2m}\longmapsto
(H_1^*,H_2^*)(t,x,p,q) \in \R$ is continuous for any fixed $(t,x)\in [0,T]\times \R^m$. 
\qed
\medskip
%
%\begin{remark}This condition has been already considered by A. Friedman in \cite{friedman} for the same purpose as ours in this paper. \qed 
%
%\end{remark}

To solve this risk-sensitive stochastic differential game, we adopt the BSDE approach. Precisely speaking, to show the game has a Nash equilibrium point, it is enough to show that its associated BSDE, which is multi-dimensional and with a  generator not standard, has a solution (see Theorem \ref{th nash} below). Therefore the main objective of the next section is to study the connection between the risk-sensitive stochastic differential game and BSDEs.
%%%%%%%%%%%%%%%%%%%%%%%%%%%
%%%%%%%%%%%%%%%%%%%%%%%%%%%
%%%%%%%%%%%%%%%%%%%%%%%%%%%
%BSDE w.r.t. u v
%%%%%%%%%%%%%%%%%%%%%%%%%%%
%%%%%%%%%%%%%%%%%%%%%%%%%%%
% %%%%%%%%%%%%%%%%%%%%%%%%%
\section{Risk-sensitive nonzero-sum stochastic differential game and BSDEs}\label{sec3}

Let $(t,x)\in \esp$ and $(\theta_s^{t,x})_{s\leq T}$ be the solution of
the following forward stochastic differential equation:

\begin{equation}\label{SDE b}
\left\{
\begin{aligned}
d\theta_s&= b(s, \theta_s)ds+ \sigma(s, \theta_s)dB_s, \qquad &s \in[t, T];\\
\theta_s&= x,  &s \in[0, t],\\
\end{aligned}
\right.
\end{equation}

\noindent where $\sigma:[0, T]\times \R^m\rightarrow \R^{m\times m}$
satisfies Assumptions (A1)(i)-(ii) and $b$: $[0, T]\times
\R^m\rightarrow \R^m$ is a measurable function which verifies the
following assumption:
\bigskip

\noindent \textbf{Assumption (A4)}: The function $b$ is uniformly Lipschitz and bounded, \textit{i.e.}, there exist constants $C_2$ and $C_b$  such that:
  $$\forall t \in [0, T],\  \forall \ x, x^{\prime} \in \R^m,\
\abs{ b(t,x)-b(t, x^{\prime})} \leq C_2 \abs{x- x^{\prime}}\mbox{ and } \abs {b(t, x)} \leq C_b.$$

Before proceeding further, let us give some useful properties of stochastic process $(\theta_s^{t,x})_{s\leq T}$.

\begin{lemma}\label{lem3.1}
Under Assumptions (A1) and (A4), we have\\
(i) the stochastic process $(\theta_s^{t,x})_{s\leq T}$ has moment of any order, \textit{i.e.} there exists a  constant $C_q \in \R$ such that: \bp-a.s.
\begin{equation}\label{estimate of theta}
\forall \ q \in [1, \infty ), \  \textbf{E}\Big[\big(\sup\limits_{s\leq
T}\abs{\theta_s^{t,x}}\big)^{2q}\Big]\leq C_q(1+\abs{x}^{2q});
\end{equation}
(ii) additionally, it satisfies the following estimate: there exists a constant $C_{\lambda,l}\in \R$, such that \bp-a.s.
\be\label{estimate of theta exp}
\forall l\in [1,2),\ \lambda\in (0, \infty),\  \textbf{E}\Big[e^{\lambda \sup\limits_{s\leq T}|\theta_s^{t,x}|^{l}}\Big]\leq e^{C_{\lambda,l}(1+|x|^{l})}. 
\ee
\no Apart from $q$, $\lambda$ and $l$, the constants $C_q$ and $C_{\lambda,l}$ in (\ref{estimate of theta})(\ref{estimate of theta exp}) depend also on  $C_b$ and $C_{\sigma}$ and $T$. 
\end{lemma}
\begin{proof}
We refer readers \cite{KS} (pp.306) for the result (i). In the following, we only provide the proof of (ii). We denote $b(s, \theta_s^{t,x})$ and $\sigma(s, \theta_s^{t,x})$ simply by $b_s$ and $\sigma_s$. Considering $(b_s)_{s\leq T}$ is bounded and $\textbf{E}[f]=\int_0^\infty \bp\{f>u\}du$ for all positive function $f$, we obtain, 
\begin{align*}
\textbf{E}&[e^{\lambda\sup_{s\leq T}|\theta_s^{t,x}|^{l}}]\\
&=\textbf{E}[e^{\lambda\sup_{s\leq T}|x+\int_t^sb_rds+\int_t^s\sigma_rdB_r|^{l}}]\\
&\leq e^{C_{l, \lambda,b, T}\cdot(1+|x|^{l})}\textbf{E}[e^{C_{l, \lambda}\cdot\sup_{s\leq T}|\int_0^s\sigma_rdB_r|^{l}}]\\
&=e^{C_{l, \lambda,b, T}\cdot(1+|x|^{l})} \int_0^{\infty}\bp\{e^{C_{l, \lambda}\cdot\sup_{s\leq T}|\int_0^s\sigma_rdB_r|^{l}}>u\}du\\
&= e^{C_{l, \lambda,b, T}\cdot(1+|x|^{l})} \left(1+\int_1^{\infty}\bp\{e^{C_{l, \lambda}\cdot\sup_{s\leq T}|\int_0^s\sigma_rdB_r|^{l}}>e^{C_{l, \lambda}\cdot u^{l}}\}de^{C_{l, \lambda}\cdot u^{l}} \right)\\
&=e^{C_{l, \lambda,b, T}\cdot(1+|x|^{l})} \left(1+\int_0^{\infty}\bp\{\sup_{s\leq T}|\int_0^s\sigma_rdB_r|>u\} e^{C_{l, \lambda}\cdot u^{l}}  C_{l, \lambda} l u^{l-1}du\right).
\end{align*}
\no Apply Theorem 2 in \cite{meyer 1970} (pp.247), $\bp\{\sup_{s\leq T}|\int_0^s\sigma_rdB_r|>u\}\leq e^{-\frac{u^2}{2TC_\sigma^2}}$. Therefore, 
\begin{align*}
\textbf{E}&[e^{\lambda\sup_{s\leq T}|\theta_s^{t,x}|^{l}}]\\
&\leq e^{C_{l, \lambda,b, T}\cdot(1+|x|^{l})}\left(1+\int_0^{\infty}e^{-\frac{u^2}{2TC_\sigma^2}}e^{C_{l,\lambda}\cdot u^{l}}  C_{l, \lambda} l u^{l-1}du\right)\\
&\leq e^{C_{l, \lambda,b, T,\sigma}\cdot(1+|x|^{l})}.
\end{align*}
\no The above inequality is finite since $1\leq l<2$ and $u\leq e^u$ for any $u>0$.
\end{proof}

Next let us recall the following result by Hausmann
(\cite{Haussmann1986}, pp.14) related to integrability of the Dol\'ean-Dade exponential local martingale defined by (\ref{density fun}).
\begin{lemma}\label{density function lp bounded}
Assume (A1)-(i)(ii) and (A4), let $(\theta^{t,x}_s)_{s\leq T}$ be the solution of
(\ref{SDE b}) and $\varphi$ be a $\mathcal{P}\otimes \mathcal{B}(\R^m)$-measurable application from $[0,T]\times \Omega \times \R^m$ to $\R^m$ which is of linear growth, that is, $\bf{P}$ -a.s., $\forall (s,x)\in [0,T]\times \R^m$,
\begin{equation*} 
\abs{\varphi(s,\omega, x)}\leq C_3(1+\abs{x}).
\end{equation*}
\noindent Then, there exists some $p\in (1,2)$ and a constant $C$, where
$p$ depends only on $C_{\sigma}$, $C_2$, $C_b$, $C_3$, $m$ while the constant $C$,
depends only on $m$ and $p$, but not on $\varphi$, such that:
\begin{equation}
\textbf{E}\left[\Big|\zeta_T (\int_0^{\cdot}\varphi(s, \theta^{t,x}_s)dB_s)\Big|^p\right] \leq
C,
\end{equation}
\noindent where the process $\zeta(\int_0^{.}\varphi(s, \theta_s^{t,x})dB_s)$ is the density function defined in (\ref{density fun}).
\end{lemma}		
\medskip

It follows from Lemma \ref{density function lp bounded} that,
\medskip

\begin{corollary}\label{coro lp}
For an admissible control $(u,v)\in \mathcal{M}$ and $(t,x)\in \esp$, there exists some $p_0\in (1,2)$ and a constant $C$, such that
\begin{equation}
\textbf{E}\left[\Big|\zeta_T (\int_0^{\cdot}\sigma(s, X_s^{t,x})^{-1}f(s, X_s^{t,x},u_s,v_s)dB_s)\Big|^{p_0}\right] \leq C.
\end{equation}
\end{corollary}	
\begin{remark}
Corollary \ref{coro lp} is needed for us in the proofs of Proposition \ref{prop BSDE uv} and Theorem \ref{th nash} which is the main result of this work.  Notice that the function $f$ is no longer bounded as in the literature but is of linear growth in $x$. 
\end{remark}

As a by-product of Lemma  \ref{lem3.1} and \ref{density function lp bounded},  we also have the similar estimates for the process $X^{t,x}$.
%----------------------------------------------------------------------------
\begin{lemma}
(i) There exist two constants $\bar{C}_q,\ \bar{C}_{\lambda,l}\in \R$, such that \bp-a.s.

\begin{equation}\label{estimate of X sigma}
\forall \ q \in [1, \infty ), \  \textbf{E}\Big[(\sup\limits_{s\leq
T}\abs{X_s^{t,x}})^{2q}\Big]\leq \bar{C}_q(1+\abs{x}^{2q}),
\end{equation}
\no and 
\be\label{estimate of X sigma exp}
\forall l\in [1,2),\ \lambda\in (0, \infty),\  \textbf{E}\Big[e^{\lambda \sup\limits_{s\leq T}|X_s^{t,x}|^{l}}\Big]\leq e^{\bar{C}_{\lambda,l}(1+|x|^{l})}; 
\ee
(ii) Moreover, for solutions of the weak formulation of SDEs (\ref{SDEf}), we have the similar results. Precisely speaking, for $(u,v)\in \mathcal{M}$, $\textbf{E}_{t,x}^{u,v}$ is the expectation under the probability $\bp_{t,x}^{u,v}$, then there exist constants $\tilde{C}_q$,$\tilde{C}_{\lambda,l}\in \R$, such that \bp-a.s.

\begin{equation}\label{estimate euv of X sigma}
\forall \ q \in [1, \infty ), \  \textbf{E}_{t,x}^{u,v}\Big[\Big(\sup\limits_{s\leq
T}\abs{X_s^{t,x}}\Big)^{2q}\Big]\leq \tilde{C}_q(1+\abs{x}^{2q}),
\end{equation}
\no and 
\be\label{estimate euv of X sigma exp}
\forall l\in [1,2),\ \lambda\in (0,\infty),\  \textbf{E}_{t,x}^{u,v}\Big[e^{\lambda \sup\limits_{s\leq T}|X_s^{t,x}|^{l}}\Big]\leq e^{\tilde{C}_{\lambda,l}(1+|x|^{l})}. 
\ee 
\end{lemma}
\begin{proof}
We only prove \eqref{estimate euv of X sigma exp}. Since,
\[
\textbf{E}_{t,x}^{u,v}\Big[e^{\lambda \sup\limits_{s\leq T}|X_s^{t,x}|^{l}}\Big]= \textbf{E}\Big[e^{\lambda \sup\limits_{s\leq T}|X_s^{t,x}|^{l}}\cdot \zeta_T\Big],
\] 
where $\zeta_T$ represents $\zeta_T (\int_0^{\cdot}\sigma(s, X_s^{t,x})^{-1}f(s, X_s^{t,x},u_s,v_s)dB_s)$. As a result of Corollary \ref{coro lp}, there exists some $p_0\in (1,2)$, such that, $\zeta_T\in L^{p_0}$. Therefore, by Young's inequality and \eqref{estimate of X sigma exp}, we obtain that,
\begin{align*}
\textbf{E}_{t,x}^{u,v}\left[e^{\lambda \sup\limits_{s\leq T}|X_s^{t,x}|^{l}}\right]
                      &\leq \textbf{E} \left[e^{\frac{p_0 \lambda}{p_0-1} \sup\limits_{s\leq T}|X_s^{t,x}|^{l}}\right]+ \textbf{E}\left[|\zeta_T|^{p_0}\right]\\
                      &\leq e^{\bar{C}_{\lambda,l, p_0}(1+|x|^{l})}+C_{m,p_0}\\
                      &\leq e^{\tilde{C}_{\lambda,l, m,p_0}(1+|x|^{l})}.                      
\end{align*} 
\end{proof}

The next proposition characterizes the payoff function $J^i(u,v)$ for $i=1,2$ with form (\ref{payoff}) by means of BSDEs. It turns out that the payoffs $J^i(u,v)$ can be expressed as the exponential of the initial value for a related BSDE. It is multidimensional, with a continuous generator involving a quadratic term of $Z$.

\begin{proposition} \label{prop BSDE uv}

Under Assumptions (A1) and (A2), for any admissible control $(u, v)\in \cal{M}$, there exists a pair of 
 adapted processes $(Y^{i,(u,v)}, Z^{i,(u,v)})$, $i=1,2$, with values on $\R\times \R^m$ such that:
\begin{description}
  \item[(i)] For any $p>1$, 
  \begin{equation}\label{estimate of Yuv under Puv}
  \euv\big[\sup_{0\leq t\leq T}e^{pY_t^{i,(u,v)}}\big]<\infty \ \text{and}\  \bp-a.s. \int_0^T|Z_t^{i,(u,v)}|^2dt<\infty.
  \end{equation}
%  Moreover,  for some $q_0\in(0,1]$, $\bar{q}_0\in (0,q_0)$ and $i=1,2$, 
%  \begin{equation}\label{estimate of Yuv and Zuv under P}
%  \textbf{E}\Big[\sup_{0\leq t\leq T}e^{q_0|Y_t^{i,(u,v)}|}+\Big(\int_0^T|Z_s^{i,(u,v)}|^2ds\Big)^{\frac{\bar{q}_0}{2}}\Big]<\infty,
%  \end{equation}

  \item[(ii)] For $t\leq T$, 
  \begin{align}\label{BSDE uv}
   Y_t^{i,(u,v)}= g^i(X_T^{0,x_0})&+\int_t^T \big\{H_i(s, X_s^{0,x_0}, Z_s^{i,(u,v)}, u_s, v_s)+\frac{1}{2}|Z_s^{i,(u,v)}|^2\big\}ds\nonumber\\
   &-\int_t^T Z_s^{i,(u,v)}dB_s.
   \end{align}
\end{description}

\noindent The solution is unique for fixed $x_0\in \R^m$. Moreover, $J^i(u,v)=e^{Y_0^{i,(u,v)}}$.
\end{proposition}

\begin{proof}
\textit{Part $\uppercase\expandafter{\romannumeral 1}$ : Existence and uniqueness.} 
We take the case of $i=1$ for example, and of course the case of $i=2$ can be solved in a similar way. The main method here is to define a reasonable form of the solution directly. We first eliminate the quadratic term in the generator by applying the classical exponential exchange. Then, the definition of $Y$ component is closely related to Girsanov's transformation, and the process $Z$ is given by  the martingale representation theorem. Afterwards, we shall verify by \ito 's formula that what we defined above is exactly the solution of the original BSDE. 
\medskip

%Let $\bp^{u,v}$ be the probability on $(\Omega, \cal{F})$ defined as follows:
%\begin{equation*}
%d\bp^{u,v}= \zeta \Big(\int_0^T \sigma^{-1}\left(s, X_s^{0,x_0}\right)f\left(s, X_s^{0,x_0}, u_s, v_s\right)\Big)d\bp,
%\end{equation*}
%\no where $(\zeta_s)_{s\leq T}$ is introduced by (\ref{density fun}). Then, 
As we stated in the previous section, the process $(X_s^{0,x_0})_{s\leq T}$ satisfies SDE \eqref{SDEf} by substituting $(0,x_0)$ for $(t,x)$.
\medskip

In order to remove the quadratic part in the generator of BSDE (\ref{BSDE uv}), we first take the classical exponential exchange as follows: $\forall t\leq T$, let
\bes
\left\{
\bal
\ybuv&=e^{Y_t^{1,(u,v)}};\\
\zbtuv&=\ybuv Z_t^{1,(u,v)}.
\eal
\right.
\ees
\no Therefore, the processes $(\ybuv,\zbtuv)_{t\leq T}$ solve the following BSDE:
\begin{align}\label{bsde uv exp}
\ybuv =e^{g^1(X_T^{0,x_0})}&+\int_t^T \zbsuv \sigma^{-1}(s, X_s^{0,x_0})f(s, X_s^{0,x_0}, u_s, v_s)\nonumber\\
&+(\bar{Y}_s^{1,(u,v)})^{+}h(s, X_s^{0,x_0}, u_s, v_s)ds- \int_t^T \zbsuv dB_s,\ t\leq T.
\end{align}
\no Applying Girsanov's transformation as indicated by \eqref{new prob}-\eqref{density fun}, the BSDE \eqref{bsde uv exp} then reduces to 
\bes
\ybuv = e^{g^1(X_T^{0,x_0})}+\int_t^T(\bar{Y}_s^{1,(u,v)})^{+}h(s, X_s^{0,x_0}, u_s, v_s)ds- \int_t^T \zbsuv d\bsuv,\ t\leq T.
\ees
Let us now define the process $\bar{Y}^{1,(u,v)}$ explicitly by: 
\be \label{def ybuv}
\ybuv:=\euv\Big[\exp\Big\{g^1(X_T^{0,x_0})+\int_t^Th_1(s, X_s^{0,x_0}, u_s, v_s)ds\Big\}\Big|\mathcal{F}_t\Big],\ t\leq T.
\ee
\no Considering the sub-quadratic growth Assumptions (A2)-(ii)(iii) on $h_1$ and $g^1$ and the estimate (\ref{estimate euv of X sigma exp}), we obtain,
\begin{align*}
\euv\Big[\exp&\Big\{g^1(X_T^{0,x_0})+\int_0^Th_1(s, X_s^{0,x_0}, u_s, v_s)ds\Big\}\Big]\\
 &\leq  \euv\Big[\exp{\Big\{C\sup_{0\leq s\leq T}\left(1+\abs{X_s^{0,x_0}}^{\gamma}\right)\Big\}}\Big]< \infty,
\end{align*}
\no with constant $C=C_g\vee(TC_h)$. Therefore, we claim that the process $(\ybuv)_{t\leq T}$ in \eqref{def ybuv} is well-defined. 
\medskip

We will give now the definition of process $(\zbtuv)_{t\leq T}$. In the following, for notation convenience, we denote by $\zeta$ the following process $\zeta:=(\zeta_t)_{t\leq T}=(\zeta_t(\int_0^.\sigma^{-1}(s, X_s^{0,x_0})f(s, X_s^{0,x_0}, u_s, v_s)dB_s))_{t\leq T}$. Then the definition (\ref{def ybuv}) can be rewritten as:
\be\label{3.14}
 \ybuv= \zeta_t^{-1}\cdot \textbf{E}\Big[\zeta_T\cdot\exp\Big\{g^1(X_T^{0,x_0})+\int_t^Th_1(s, X_s^{0,x_0}, u_s, v_s)ds\Big\}\Big|\mathcal{F}_t\Big],\ t\leq T.
\ee 
Thanks to Corollary \ref{coro lp}, there exists some $p_0\in (1,2)$, such that {\small{$\textbf{E}[|\zeta_T|^{p_0}]<\infty$}}. Therefore, from Young's inequality, we get that for any constant $q\in(1,p_0)$,
\begin{equation*}
\begin{aligned}
&\quad \textbf{E} \Big[\Big|\zeta_T\cdot \exp{\Big\{g^1(X_T^{0,x_0})+\int_0^T
h_1(s,
X_s^{0,x_0},u_s,v_s)ds\Big\}}\Big|^q\Big]\\
&\leq \frac{q}{p_0}\textbf{E}\left[\abs{\zeta_T}^{p_0}\right]+
\frac{p_0-q}{p_0}\textbf{E}\Big[\exp{\Big\{\frac{qp_0}{p_0-q}\cdot\big(g^1(X_T^{0,x_0})+\int_0^Th_1(s,
X_s^{0,x_0},u_s,v_s)ds\big)\Big\}}\Big].
\end{aligned} 
\end{equation*}

\no As a consequence of Assumptions (A2)-(ii)(iii) and (\ref{estimate of theta exp}), the following expectation is finite, i.e.,
\begin{equation*}
\frac{p_0-q}{p_0}\textbf{E}\Big[\exp{\Big\{\frac{qp_0}{p_0-q}\cdot\big(g^1(X_T^{0,x_0})+\int_0^Th_1(s,
X_s^{0,x_0},u_s,v_s)ds\big)\Big\}}\Big]<\infty
\end{equation*}
\no Then, we deduce that, 
\begin{equation*}
\zeta_T\cdot \exp{\Big\{g^1(X_T^{0,x_0})+\int_0^T
h_1(s,
X_s^{0,x_0},u_s,v_s)ds\Big\}}\in \mathcal{L}^q(d\bp).
\end{equation*}

\no It follows from \eqref{3.14} and the representation theorem that, there exists a $\mathcal{P}$-measurable process $(\bar{\theta}_s)_{s\leq T}\in \mathcal{H}_T^q(\R^m)$, such that for any $t\leq T$,

\begin{equation*}
\begin{aligned}
\ybuv&= \zeta_t^{-1}\exp\{-\int_0^th_1(s, X_s^{0,x_0}, u_s, v_s)ds\}\times\\
&\quad\times\{\textbf{E}\Big[\zeta_T \exp\Big\{g^1(X_T^{0,x_0})+\int_0^Th_1(s, X_s^{0,x_0}, u_s, v_s)ds\Big\}\Big]+\int_0^t\bar{\theta}_sdB_s\}
\end{aligned}
\end{equation*}

\no Let us denote by:
\bes
R_t:= \textbf{E}\Big[\zeta_T \exp\Big\{g^1(X_T^{0,x_0})+\int_0^Th_1(s, X_s^{0,x_0}, u_s, v_s)ds\Big\}\Big]+\int_0^t\bar{\theta}_sdB_s,\quad t\leq T.
\ees

\no Taking account of $d\zeta_t = \zeta_t \sigma^{-1}(t, X_t^{0,x_0})f(t, X_t^{0,x_0}, u_t, v_t)dB_t$ for $t\leq T$, then by \ito's formula, we have 
$d\zeta_t^{-1} = -\zeta_t^{-1}\{\sigma^{-1}(t, X_t^{0,x_0})f(t, X_t^{0,x_0}, u_t, v_t)dB_t-|\sigma^{-1}(t, X_t^{0,x_0})f(t, X_t^{0,x_0}, u_t, v_t)|^2dt\}, \ t\leq T.$
\no Moreover,
\begin{align*}
d\Big[\zeta_t^{-1}&\exp\{-\int_0^th_1(s,X_s^{0,x_0},u_s,v_s)ds\}\Big]\\
&=-\zeta_t^{-1}\exp\{-\int_0^th_1(s,X_s^{0,x_0,u_s,v_s})ds\}\Big\{\sigma^{-1}(t, X_t^{0,x_0})f(t, X_t^{0,x_0}, u_t, v_t)dB_t\\
&\quad+\big[-|\sigma^{-1}(t, X_t^{0,x_0})f(t, X_t^{0,x_0}, u_t, v_t)|^2+h_1(t,X_t^{0,x_0},u_t,v_t)\big]dt\Big\}, \ t\leq T.
\end{align*}
\no Hence, for $t\leq T$, 
\begin{align*}
d\ybuv = &-\zeta_t^{-1}\exp\{-\int_0^th_1(s,X_s^{0,x_0},u_s,v_s)ds\}\Big\{\sigma^{-1}(t, X_t^{0,x_0})f(t, X_t^{0,x_0}, u_t, v_t)dB_t\\
&\qquad\qquad+\big[-|\sigma^{-1}(t, X_t^{0,x_0})f(t, X_t^{0,x_0}, u_t, v_t)|^2+h_1(t,X_t^{0,x_0},u_t,v_t)\big]dt\Big\}R_t\\
&+ \zeta_t^{-1}\exp\{-\int_0^th_1(s,X_s^{0,x_0},u_s,v_s)ds\}\bar{\theta}_tdB_t\\
&-\zeta_t^{-1}\exp\{-\int_0^th_1(s,X_s^{0,x_0},u_s,v_s)ds\}\sigma^{-1}(t, X_t^{0,x_0})f(t, X_t^{0,x_0}, u_t, v_t)\bar{\theta}_tdt,
\end{align*}

\no which allows us to define the process $\bar{Z}^{1,(u,v)}$ as the volatility coefficient of the above equation, \textit{i.e.}, for $t\leq T$,
\begin{align}\label{def zbtuv}
 \zbtuv:= -\zeta_t^{-1}\exp\{-\int_0^th_1(s,X_s^{0,x_0},u_s,v_s)ds\}\Big\{&\sigma^{-1}(t, X_t^{0,x_0})f(t, X_t^{0,x_0}, u_t, v_t)R_t\nonumber\\
 &-\bar{\theta}_t\Big\}.
\end{align} 

\no Then, it is not difficult to verify that the process $(\ybuv, \zbtuv)_{t\leq T}$, as we defined by (\ref{def ybuv}) (\ref{def zbtuv}) satisfies the BSDE (\ref{bsde uv exp}). Moreover, it can be seen obviously from (\ref{def ybuv}) that $\ybuv>0$ for all $t\in [0,T]$. Therefore, we define the pair of processes $(Y^{1,(u,v)},Z^{1,(u,v)})$ as follows: 
\bes
\left\{
\bal
\ytuv &= \ln \ybuv;\\
\ztuv &= \frac{\zbtuv}{\ybuv},\quad t\leq T.
\eal
\right.
\ees  
\no which completes the proof of existence. 
\medskip

The uniqueness is natural by the above construction itself for fixed $x_0\in \R^m$. Since, the solution of BSDE \eqref{bsde uv exp}, if exists, will be of the form \eqref{def ybuv} and \eqref{def zbtuv}.
\medskip

\textit{Part $\uppercase\expandafter{\romannumeral 2}$ : Norm estimates.}\quad  Finally, let us focus on the estimate of $(\ytuv)_{t\leq T}$ which is needed in the next theorem. First, as a consequence of the definition (\ref{def ybuv}) that for any $p>1$,
\begin{equation}\label{3.17}
\begin{aligned}
\euv&\big[|\sup_{t\in[0,T]}\ybuv|^p\big]\\
&\leq \euv \Big[\Big|\sup_{t\in{[0,T]}}\Big\{\euv\big[\exp \{g^1(X_T^{0,x_0})+\int_0^T|h_1(s, X_s^{0,x_0}, u_s, v_s)|ds\}|\mathcal{F}_t\big]\Big\}\Big|^p\Big].
\end{aligned}
\end{equation}
\no Noticing that the process {\small{$\euv\big[\exp \{g^1(X_T^{0,x_0})+\int_0^T|h_1(s, X_s^{0,x_0}, u_s, v_s)|ds)\}|\mathcal{F}_.\big]$}}  is a $\mathcal{F}_t$-martingale, then Doob's maximal inequality(see \cite{KS} pp.14) implies that,
\begin{equation}
\begin{aligned}
\euv&\big[|\sup_{t\in[0,T]}\ybuv|^p\big]\\
&\leq (\frac{p}{p-1})^p\euv\Big[\Big|\euv\big[\exp \{g^1(X_T^{0,x_0})+\int_0^T|h_1(s, X_s^{0,x_0}, u_s, v_s)|ds\}|\mathcal{F}_T\big]\Big|^{p}\Big]
\end{aligned}
\end{equation}
\no Then, considering the Jensen's inequality and Assumption (A2)(ii)-(iii)  on $g^1$ and $h_1$, it turns out that,
\begin{equation}
\begin{aligned}
\euv&\big[|\sup_{t\in[0,T]}\ybuv|^p\big]\\
&\leq (\frac{p}{p-1})^p\euv\big[\exp \{pg^1(X_T^{0,x_0})+p\int_0^T|h_1(s, X_s^{0,x_0}, u_s, v_s)|ds\}\big]\\
&\leq (\frac{p}{p-1})^p\euv\big[e^{\sup\limits_{t\in[0,T]}C(1+|X_t^{0,x_0}|^{\gamma})}\big]< \infty,
\end{aligned}
\end{equation}
which is given by the estimate (\ref{estimate euv of X sigma exp}) with constant $C$ depending on $p,\ C_g,\ C_h,$ and $T$. Therefore, 
\begin{equation}\label{3.20}
\euv\big[\sup_{t\in[0,T]}|\ybuv|^p\big]<\infty,
\end{equation}
\no which gives, 
\bes
\euv\big[\sup_{t\in [0,T]}e^{p\ytuv}\big]<\infty, \qquad \forall p>1.
\ees
%%%%%%%%%%%%%%%%%%%%%%%%%%%%%%%%%%%

At last, note that in taking $t=0$ in (\ref{def ybuv}) we obtain {\small{$J^1(u,v)=\bar{Y}_0^{1,(u,v)}=e^{Y_0^{1,(u,v)}}$}} since $\mathcal{F}_0$  contains only $\bp$ and $\bp^{u,v}$ null sets.
\end{proof}
\medskip

We are now ready to demonstrate the existence of Nash equilibrium point which is the main result of this article. 
\medskip

\begin{theorem}\label{th nash} Let us assume that:

\no (i) Assumptions (A1), (A2) and (A3) are fulfilled ;

\no (ii) There exist two pairs of $\mathcal{P}$-measurable processes $(Y^i,Z^i)$ with values in $\R^{1+m}$, $i=1,2$, and two deterministic functions $\varpi^i(t,x)$ which are of subquadratic growth, i.e. $|\varpi^i(t,x)|\leq C(1+|x|^{\gamma})$ with $1<\gamma<2$, $i=1,2$ such that:
\begin{equation}\label{main BSDE}
\left\{
\begin{aligned}
\bp&\text{-a.s.}, \forall t\leq T, Y_t^i=\varpi^i(t,X_t^{0,x}) \text{ and } Z^i \text{is dt-square integrable } \bp\text{-a.s.};\\
Y^i_t &=g^i(\xTo)+\int_t^T \{H_i(s, \xso, Z_s^i, (\us, \vs)(s, \xso, Z_s^1, Z_s^2))+\frac{1}{2}\abs{Z_s^i}^2\}ds\\
&\qquad \qquad \quad- \int_t^T Z_s^i dB_s,\ \forall t\leq T.
\end{aligned}
\right.
\end{equation}

\no Then the pair of control $(u^*(s, X_s^{0,x}, Z_s^1, Z_s^2),v^*(s,X_s^{0,x}, Z_s^1, Z_s^2))_{s\leq T}$ is admissible and a Nash equilibrium point for the game. 
\end{theorem}
\begin{proof}
For $s\leq T$, let us set $u_s^*= u^*(s,\xso,Z_s^1,Z_s^2)$ and $v_s^*=v^*(s,\xso,Z_s^1,Z_s^2)$. Then $(u^*, v^*)\in \mathcal{M}$. On the other hand, we obviously have $J^1(u^*, v^*)=e^{Y_0^1}$ by Proposition \ref{prop BSDE uv}. Next for an arbitrary element $u\in \mathcal{M}_1$, let us show that $e^{Y^1}\leq e^{Y^{u,v^*}}$, which yields $e^{Y_0^1}=J^1(u^*,v^*)\leq J^1(u, v^*)=e^{Y_0^{1,(u,v^*)}}$. We focus on this point below. For the admissible control $(u, v^*)$, thanks to Proposition \ref{prop BSDE uv}, there exists a pair of $\mathcal{P}$-measurable processes $(Y_t^{i,(u,\vs)}, Z_t^{i,(u,\vs)})_{t\leq T}$ for $i=1,2$, which satisfies: for any $p>1$,
\be\label{BSDE uvs} 
\left\{
\begin{aligned}
Y^{i,(u,\vs)} &\in \mathcal{D}_T^p(\R, d\bp^{u, v^*}), \  Z^{i,(u,\vs)} \text{ is }dt\text{-square integrable }\bp\text{-a.s.}\\
Y_t^{i,(u,\vs)}&= g^i(\xTo)+\int_t^T \{H_i(s, \xso, Z_s^{i,(u,\vs)}, u_s, v_s^*)+\frac{1}{2}|Z_s^{i,(u,\vs)}|^2 \}dt\\
&\qquad\qquad\quad - \int_t^T Z_s^{i,(u,\vs)}dB_s, \ \forall t\leq T.
\end{aligned}
\right.
\ee

Let us set: $\forall t\leq T$,
\bes
\dtusvs := e^{Y_t^1},\qquad \dtuvs := e^{\ytuvs}.
\ees
\no Thus \ito -Meyer formula yields, for any $t\leq T$,
\begin{align}\label{3.22}
-&d(\dtusvs -\dtuvs)^+ +dL_t^0(D^{u^*, v^*}-D^{u, v^*}) \nonumber\\
&=\left[\dtusvs H_1(t, \xto, Z_t^1, u_t^*, v_t^*)- \dtuvs H_1(t, \xto, \ztuvs, u_t, v_t^*)\right]1_{\{\dtusvs-\dtuvs>0\}}dt\nonumber\\
&\quad - (\dtusvs Z_t^1- \dtuvs \ztuvs) 1_{\{\dtusvs-\dtuvs>0\}}dB_t\nonumber\\
&= \Big[\dtusvs H_1(t, \xto, Z_t^1, u_t^*, v_t^*)- \dtusvs H_1(t, \xto, Z_t^1, u_t, v_t^*)\nonumber\\
&\quad +\dtusvs H_1(t, \xto, Z_t^1, u_t, v_t^*)-\dtuvs H_1(t, \xto, \ztuvs, u_t, v_t^*)\Big]1_{\{\dtusvs-\dtuvs>0\}}dt\nonumber\\
&\quad -(\dtusvs Z_t^1-\dtuvs\ztuvs)1_{\{\dtusvs-\dtuvs>0\}}dB_t\nonumber\\
&=\Big[\dtusvs\left( H_1(t, \xto, Z_t^1, u_t^*, v_t^*)-H_1(t, \xto, Z_t^1, u_t, v_t^*)\right)\nonumber\\
&\quad + (\dtusvs-\dtuvs)^+ h_1(t, \xto, u_t, v_t^*)\nonumber\\
&\quad +(\dtusvs Z_t^1-\dtuvs\ztuvs)\sigma^{-1}(t, \xto)f(t,\xto, u_t, v_t^*)\Big]1_{\{\dtusvs-\dtuvs>0\}}dt\nonumber\\
&\quad -(\dtusvs Z_t^1-\dtuvs\ztuvs)1_{\{\dtusvs-\dtuvs>0\}}dB_t,
\end{align}
where $L_t^0=L_t^0(D^{u^*, v^*}-D^{u, v^*})$ is the local time of the continuous semimartingale $D^{u^*, v^*}-D^{u, v^*}$ at time $0$. Next for $t\leq T$, let us  give $B_t^{u,\vs}=(B_t-\int_0^t\sigma^{-1}(s, \xso)f(s,\xso, u_s, v_s^*)ds)_{t\leq T}$ which is an $\mathcal{F}_t$-Brownian motion under the probability $\bp^{u,v^*}$, whose density w.r.t. $\bp$ is defined by $\zeta_T:=\zeta_T(\int_0^{.}\sigma^{-1}(s,X_s^{0,x})f(s,X_s^{0,x},u_s,v_s^*)dB_s)$ as defined in (\ref{new prob}). On the other hand, let us denote:
\bes
\Gamma_t^1 :=(\dtusvs-\dtuvs)^+\exp\{\int_0^t h_1(s, \xso, u_s, v_s^*)ds\},\qquad t\leq T.
\ees

\no Taking into account of \eqref{3.22}, we then conclude by It\^o's formula and Girsanov's transformation that, for $t\leq T$,
\begin{align} \label{Gamma}
d\Gamma_t^1 &= -\exp\{\int_0^t h_1(s, X_s^{0,x}, u_s, v_s^*)ds\}\times \nonumber \\
&\qquad \times\left[\dtusvs\Delta_t^1 dt-(\dtusvs \!Z_t^1 -\dtuvs\ztuvs)dB_t^{u, v^*}-dL_t^0 \right],
\end{align}
\no where 
\bes
\Delta_t^1= H_1(t, \xto, Z_t^1, u_t^*, v_t^*)- H_1(t, \xto, Z_t^1, u_t, v_t^*) \leq 0,
\ees
\no which is obtained by the generalized Isaacs' Assumption (A3)-(i). Next, let us define the stopping time $\tau_n$ as follows:
\begin{equation*}
\tau_n= \inf\{t\geq 0, |D_t^{u,v^*}|+|D_t^{u^*,v^*}|+\int_0^t ( |Z_s^1|^2+|Z_s^{1,(u, 
\vs)}|^2)ds \geq n\}\wedge T. 
\end{equation*} 
\no The sequence of stopping times $(\tau_n)_{n\geq 0}$ is of stationary type and converges to $T$ as $n\rightarrow \infty$. We then claim that, $\int_0^{t\wedge \tau_n}\! \exp\{\int_0^s h_1(r, \xro, u_r, v_r^*)dr\}1_{\{\dsusvs-\dsuvs>0\}}(\dsusvs Z_s^1\!-\!\dsuvs\zsuvs)dB_s^{u, \vs}$ is a $\mathcal{F}_t$-martingale under the probability $\bp^{u,v^*}$ as the following expectation
\begin{align}\label{3.24}
&\textbf{E}^{u, \vs }\left[\int_0^{\tau_n}e^{2\int_0^s h_1(r, \xro, u_r, v_r^*)dr} (\dsusvs Z_s^1-\dsuvs\zsuvs)^2ds\right]\nonumber\\
&\leq \textbf{E}^{u, \vs }\left[\int_0^{\tau_n}e^{2\int_0^s h_1(r, \xro, u_r, v_r^*)dr} \left(2 |\dsusvs|^2 |Z_s^1|^2+ 2|\dsuvs|^2 |\zsuvs|^2\right)ds\right]\nonumber\\
&\leq \textbf{E}^{u, \vs }\left[\sup_{0\leq s\leq \tau_n} \left\{2e^{2C_h(1+|X_s^{0,x}|^\gamma)} |\dsusvs|^2 \right\} \cdot \int_0^{\tau_n} |Z_s^1|^2\right]+ \nonumber\\
&\quad \textbf{E}^{u, \vs }\left[\sup_{0\leq s\leq \tau_n} \left\{2e^{2C_h(1+|X_s^{0,x}|^\gamma)} |\dsuvs|^2 \right\} \cdot \int_0^{\tau_n} |Z_s^{1, (u,v^*)}|^2\right]
\end{align}

\no is finite which is the consequence of the definition of $\tau_n$ and the estimate (\ref{estimate euv of X sigma exp}). Considering that $L_t^0$ is an increasing process, therefore, $\int_{t\wedge \tau_n}^{\tau_n}\exp\{\int_0^s h_1(r, X_r^{0,x}, u_r, v_r^*)dr\} dL_s^0$ is positive. Now returning to equation (\ref{Gamma}), then taking integral on interval $(t\wedge \tau_n, \tau_n)$ and conditional expectation w.r.t. $\mathcal{F}_{t\wedge \tau_n}$ under the probability $\bp^{u, v^*}$,  yield that,
\[
\Gamma_{t\wedge \tau_n}^1 \leq \textbf{E}^{u,\vs}\left[\Gamma^1_{\tau_n}\Big|\mathcal{F}_{t\wedge \tau_n}\right],
\]
i.e.,
\be \label{t tau n}
\textbf{E}^{u,\vs}\Gamma_{t\wedge \tau_n}^1\leq \textbf{E}^{u,\vs}\Gamma_{\tau_n}^1.
\ee  
\no Indeed, for any $p>1$, $1<q<p$, and given $1<\gamma<2$, we have,
\begin{align}\label{gamma tau n}
\textbf{E}^{u,\vs}&\left[\sup_{0\leq t\leq T}\abs{\Gamma^1_t}^q\right]\nonumber\\ 
&=\textbf{E}^{u,\vs}\Big[\sup_{0\leq t\leq T}\Big\{|D_t^{\us,\vs}-D_t^{u,\vs}|^q \exp\{q\int_0^th_1(s, \xso, u_s, v_s^*)ds\}\Big\}\Big]\nonumber\\
&\leq C\{\textbf{E}^{u,\vs}\Big[\sup_{0\leq t\leq T}e^{pY_t^1}+\sup_{0\leq t\leq T}e^{pY_t^{1,(u, \vs)}}\Big]\nonumber\\
&\quad + \textbf {E}^{u,\vs}\Big[\sup_{0\leq t\leq T}e^{q\cdot \frac{p}{p-q}C_h(1+|\xto|^{\gamma})}\Big]\}.
\end{align}
\no Indeed, for any $p>1$, $Y^1\in \mathcal{D}_T^p(\R, d\bp^{u,v^*})$, since we assume  $Y_t^1=\varpi^1(t, X_t^{0,x})$ where $\varpi^1$ is deterministic and of subquadratic growth and finally \eqref{estimate euv of X sigma exp}. Meanwhile, $Y^{1,(u,v^*)}\in \mathcal{D}_T^p(\R, d\bp^{u,v^*})$ by \eqref{BSDE uvs}.  Therefore, \eqref{gamma tau n} is finite. As the sequence $(\Gamma_{\tau_n}^1)_{n\geq 1}$ converges to $\Gamma_T^1=0$ as $n\rightarrow \infty$, ${\bf{P}}^{u,v^*}$-a.s., it then also converges to $0$ in $\mathcal{L}^{1}(d\bp^{u, \vs})$ since it is uniformly integral. 

Next, by passing $n$ to the limit on both sides of (\ref{t tau n}) and using the Fatou's lemma, we are able to show $\textbf{E}^{u,v^*}[\Gamma_t^1]= 0,\ \forall t\leq T$, which implies $e^{Y_t^1}\leq e^{{Y_t}^{u, \vs}}$, $\bp$-a.s., since the probabilities $\bp^{u, v^*}$ and $\bp$ are equivalent. Thus, $e^{Y_0^1}=J^1(u^*, v^*)\leq e^{Y_0^{1,(u, v^*)}}=J^1(u, v^*)$. In the same way, we can show that for arbitrary element $v\in \mathcal{M}_2$, then, $e^{Y_0^2}=J^2(u^*, v^*)\leq e^{Y_0^{2,(u^*, v)}}=J^2(u^*, v)$, which indicate that, $(\us, \vs)$ is an equilibrium point of the game. 
\end{proof}
\medskip

\section{Existence of solutions for markovian BSDE}

In Section \ref{sec3}, we provide the existence of the Nash equilibrium point under appropriate conditions. It remains to show that the BSDEs \eqref{main BSDE} have solutions as desired in Theorem \ref{th nash}. Therefore, in this section, we focus on this objective.
\medskip

We firstly recall the notion of domination. 
\subsection{Measure domination}

\begin{definition}\textbf{: $\mathcal{L}^{q}$-Domination condition}
\\
Let $q\in ]1,\infty[$ be fixed. For a given $t_1\in [0,T]$, a family of probability measures
$\{\nu_1(s,dx), s\in [t_1, T]\}$ defined on $\R^m$ is said to be $\mathcal{L}^{q}$-
dominated by another family of probability measures
$\{\nu_0(s,dx), s\in [t_1, T]\}$, if for any $\delta \in(0, T-t_1]$,
there exists an application $\phi^\d_{t_1}: [t_1+\d, T]\times \R^m \rightarrow \R^+$ such that:
\\

(i) $\nu_1$(s, dx)ds= $\phi^\d_{t_1}$(s, x)$\nu_0$(s, dx)ds on $[t_1+\delta, T]\times$
  $\R^m$.
  
(ii) $\forall k\geq 1$, $\phi^\d_{t_1}(s,x) \in \mathcal{L}^q([t_1+\delta, T]\times [-k, k]^m$; $\nu_0(s,
  dx)ds)$.\qed
\end{definition}\label{definition of lp dominated}
\medskip

We then have:

\begin{lemma}\label{aronson's estimate}
\no Let $q\in ]1,\infty[$ be fixed, $(t_0,x_0)\in \esp$ and let $(\theta_s^{t_0,x_0})_{t_0\leq s\leq T}$
be the solution of SDE (\ref{SDE b}). If the diffusion coefficient function $\sigma $ satisfies (\ref{horm}), then for any $s\in (t_0,T]$, the law $\bar \mu(t_0, x_0; s, dx)$ of
$\theta_s^{t_0,x_0}$ has a density function $\rho_{t_0,x_0}(s,x)$, w.r.t. Lebesgue measure $dx$, which satisfies the following estimate: $\forall (s,x)\in (t_0,T]\times \R^m$, 
\begin{equation}\label{aron-est}
\varrho_1(s-t_0)^{-\frac{m}{2}}exp\left[-\frac{\Lambda\abs{x-x_0}^2}{s-t_0}\right]\leq
\rho_{t_0,x_0}(s,x)\leq
\varrho_2(s-t_0)^{-\frac{m}{2}}exp\left[-\frac{\lambda\abs{x-x_0}^2}{s-t_0}\right]
\end{equation}
\no where $\varrho_1$, $\varrho_2$, $\Lambda$, $\lambda$ are real constants such that $\varrho_1 \leq \varrho_2$ and $\Lambda > \lambda$. Moreover for any $(t_1,x_1)\in [t_0,T]\times \R^m$, the family of laws $\{\bar\mu(t_1,x_1;s,dx), s\in [t_1,T]\}$ is
$L^q$-dominated by $\bar\mu(t_0,x_0;s,dx)$.
\end{lemma}
\begin{proof}
Since $\sigma$ satisfies (\ref{horm}) and $b$ is bounded, then by Aronson's result (see \cite{aronson}),
the law $\bar\mu(t_0, x_0; s, dx)$ of
$\theta_s^{t_0,x_0}$, $s\in ]t_0,T]$, has a density function $\rho_{t_0,x_0}(s,x)$ which satisfies estimate (\ref{aron-est}). 
\ms

Let us focus on the second claim of the lemma. Let $(t_1,x_1)\in [t_0,T]\times \R^m$ and $s\in (t_1,T]$. Then 
\begin{equation*}
\rho_{t_1,x_1}(s,x)=[\rho_{t_1,x_1}(s,x)\rho^{-1}_{t_0,x_0}(s,x)]\rho_{t_0,x_0}(s,x)=\phi_{t_1}(s,x)\rho_{t_0,x_0}(s,x)
\end{equation*}
\no with
$$\phi_{t_1,x_1}(s,x)=\left[\rho_{t_1,x_1}(s,x)\rho^{-1}_{t_0,x_0}(s,x)\right],
(s,x)\in (t_1,T]\times \R^m.$$
For any $\d\in (0,T-t_1]$, $\phi_{t_1,x_1}$ is defined on $[t_1+\d,T]$. Moreover for any $(s,x)\in [t_1+\d,T]$ it holds 
$$\begin{array}{ll}
\bar\mu(t_1,x_1;s,dx)ds&=\rho_{t_1,x_1}(s,x)dxds\\{}&=\phi_{t_1,x_1}(s,x)
\rho_{t_0,x_0}(s,x)dxds\\{}&=\phi_{t_1,x_1}(s,x)\bar\mu(t_0,x_0;s,dx)ds.
\end{array}$$
Next by (\ref{aron-est}), for any $(s,x)\in [t_1+\d,T]\times \R^m$,
\begin{equation*} 
0\le \phi_{t_1,x_1}(s,x)\leq
\frac{\varrho_2(s-t_1)^{-\frac{m}{2}}}{\varrho_1(s-t_0)^{-\frac{m}{2}}}exp\left[\frac{\Lambda\abs{x-x_0}^2}{s-t_0}-\frac{\lambda\abs{x-x_1}^2}{s-t_1}\right]\equiv \Phi_{t_1,x_1}(s,x).
\end{equation*}
It follows that for any $k\geq 0$, the function $\Phi_{t_1,x_1}(s,x)$ is bounded on $[t_1+\d,T]\times [-k,k]^m$ by a constant $\kappa$ which depends on $t_1$, $\d$, $\Lambda$, $\lambda$ and $k$. Next let $q\in (1,\infty)$, then 
$$\begin{array}{ll}
\int_{t_1+\d}^T\int_{[-k,k]^m}\Phi(s,x)^q\bar\mu(t_0,x_0;s,dx)ds&\leq \kappa^q \int_{t_1+\d}^T\int_{[-k,k]^m}\bar\mu(t_0,x_0;s,dx)ds\\{}&= \kappa^q\int_{t_1+\d}^Tds \textbf{E}[1_{[-k,k]^m}(\theta_s^{t_0,x_0})]\leq \kappa^q T.
\end{array}
$$
Thus $\Phi$ and then $\phi$ belong to $\mathcal{L}^q([t_1+\delta, T]\times [-k, k]^m$; $\nu_0(s,
  dx)ds)$. It follows that the family of measures 
$\{\bar\mu(t_1,x_1;s,dx), s\in [t_1,T]\}$ is
$\mathcal{L}^q$-dominated by $\bar\mu(t_0,x_0;s,dx)$. 
\end{proof} 
\ms 
  
As a by-product we have:
\begin{corollary}\label{dominated cor}
Let $x\in \R^m$ be fixed, $t\in [0,T]$, $s\in (t,T]$ and 
$\mu(t,x;s,dy)$ the law of $X^{t,x}_s$, \textit{i.e.}, 
$$\forall A \in \mathcal{B}(\R^m),\,\,
\mu(t,x;s,A)= \bp(X_s^{t,x}\in A).$$ If $\sigma$ satifies (\ref{horm}), then for any $q\in (1,\infty)$, the family 
of laws $\{\mu(t,x;s,dy), s\in [t,T]\}$ is $\mathcal{L}^q$-dominated by 
$\{\mu(0,x;s,dy), s\in [t,T]\}$.\qed
\end{corollary} 

\subsection{Existence of solutions for BSDE \eqref{main BSDE}}
Now, we are well-prepared to provide the existence of solution for BSDE (\ref{main BSDE}).
\begin{theorem}
Let $x\in \R^m$ be fixed. Then under Assumptions (A1)-(A3), there exist two pairs of $\mathcal{P}$-measurable processes $(Y^i, Z^i)$ with values in $\R^{1+m}$,  $i=1,2,$ and two deterministic functions $\varpi^i(t,x)$ which are of subquadratic growth, i.e. $|\varpi^i(t,x)|\leq C(1+|x|^{\gamma})$ with $1<\gamma<2$, $i=1,2$ such that, 
\begin{equation}\label{main BSDE in th}
\left\{
\begin{aligned}
\bp&\text{-a.s.}, \forall t\leq T, Y_t^i=\varpi^i(t,X_t^{0,x}) \text{ and } Z^i \text{is dt-square integrable } \bp\text{-a.s.};\\  
Y_t^i&=g^i(\xto)+\int_t^T\{H_i(s, \xso, Z_s^i, (\us, \vs)(s, \xso, Z_s^1, Z_s^2))+\frac{1}{2}\abs{Z_s^i}^2\}ds \\
&\quad -\int_t^T Z_s^idB_s, \quad \forall t\leq T.
\end{aligned} 
\right.
\end{equation} 
\end{theorem}

\begin{proof} We shall divide the proof into several steps. Our plan is the following. We apply the exponential exchange (see e.g. \cite{kobylanski}) to eliminate the quadratic term in the generator. The pair of the solution processes (resp. the generator) is denoted by $(\bar{Y}, \bar{Z})$ (resp. $G$). We then approximate the new generator $G$ by the Lipschitz continuous ones, which we denoted by $G^n$, such that the classical results about BSDE can be applied. It follows that, for each $n$, the BSDE with generator $G$ being replaced by $G^n$, has a solution $(\bar{Y}^n, \bar{Z}^n)$. After that, we give the uniform estimates of the solutions, as well as the convergence property. In the convergence step, the measure domination property Corollary \ref{dominated cor} plays a crucial role in passing from the weak limit to a strong sense one . Finally, we verify that the limits of the sequences are exactly the solutions of the BSDE. 
\paragraph{\emph{Step 1.}}\label{step1} \textit{Exponential exchange and approximation.}
\medskip

For $t\in [0, T]$, and $i=1,2,$ let us denote by:
\be\label{exp exchange}
\left\{
\begin{aligned}
\bar{Y}_t^i&=e^{Y_t^i};\\
\bar{Z}_t^i&=\bar{Y}_t^iZ_t^i.
\end{aligned}
\right.
\ee
\no Then, BSDE (\ref{main BSDE in th}) reads, for $t\in [0,T]$ and $ i=1,2$,
\begin{align}\label{BSDE exp}
\bar{Y}_t^i=e^{g^i(\xTo)}&+\int_t^T\mathbbm{1}_{\bar{Y}_s^i>0}\{\bar{Z}_s^i \sigma^{-1}(s, X_s^{0,x})f(s, \xso, (\us, \vs)(s, \xso, \frac{\bar{Z}_s^1}{\bar{Y}_s^1}, \frac{\bar{Z}_s^2}{\bar{Y}_s^2}))\nonumber\\
 &+ \bar{Y}_s^i h_i(s, \xso, (\us, \vs)(s, \xso, \frac{\bar{Z}_s^1}{\bar{Y}_s^1}, \frac{\bar{Z}_s^2}{\bar{Y}_s^2}))\}ds- \int_t^T \bar{Z}_s^idB_s.
\end{align}
Let us deal with the case $i=1$ for example and the case $i=2$  follows in the same way.  Inspiring by the mollify technique in \cite{hamadene1997},  we first denote here the generator of \eqref{BSDE exp} by $G^1:[0,T]\times\R^m\times \R^{+*}\times\R^{+*}\times\R^{2m}\longrightarrow\R$ (by $\R^{+*}$, we refer to $\R^+\backslash\{0\}$ ), \textit{i.e.}
$$
\begin{array}{ll}
G^1(s,x,y^1,y^2,z^1,z^2)&=\mathbbm{1}_{y^1>0}\{z^1\sigma^{-1}(s,x)f(s,x,(\us, \vs)(s,x,\frac{z^1}{y^1},\frac{z^2}{y^2}))\\
&\quad+y^1h(s,x,(\us,\vs)(s,x,\frac{z^1}{y^1}, \frac{z^2}{y^2}))\}
\end{array}
$$
\no which is still continuous w.r.t $(y^1,y^2,z^1,z^2)$ considering the Assumption (A3)-(ii) and the transformation \eqref{exp exchange}. Let $\xi$ be an element of $C^{\infty}(\R^{+*}\times\R^{+*}\times\R^{2m}, \R)$ with compact support and satisfying:
\bes
\int_{\R^{+*}\times\R^{+*}\times\R^{2m}}\xi(y^1,y^2,z^1,z^2)dy^1dy^2dz^1dz^2=1.
\ees
\no For $(t,x,y^1,y^2,z^1,z^2)\in [0,T]\times \R^m\times \R^{+*}\times\R^{+*}\times\R^{2m}$, we set,
\begin{align*}
\tilde{G}^{1n}\left(t,x,y^1,y^2,z^1,z^2\right)&=\int_{\R^{+*}\times \R^{+*}\times \R^{2m}}n^4G^1(s, \varphi_n(x),p^1,p^2,q^1,q^2)\cdot \\
&\xi\left(n(y^1-p^1),n(y^2-p^2),n(z^1-q^1),n(z^2-q^2)\right)dp^1dp^2dq^1dq^2,
\end{align*}
\no where $\varphi_n(x)=((x_j\vee (-n))\wedge n)_{j=1,2,...,m}$, for $x=(x_j)_{j=1,2,...,m}\in \R^m$.
\no We next define $\psi\in C^{\infty}(\R^{2+2m},\R)$ by,
\bes
\psi(y^1,y^2,z^1,z^2)=
\left\{
\begin{aligned}
1,\quad \abs{y^1}^2+\abs{y^2}^2+\abs{z^1}^2+\abs{z^2}^2\leq 1,\\
0,\quad \abs{y^1}^2+\abs{y^2}^2+\abs{z^1}^2+\abs{z^2}^2\geq 4.
\end{aligned}
\right.
\ees
\no Then, we define the measurable function sequence $(G^{1n})_{n\geq 1}$ as follows: $\forall (t,x,y^1,y^2,z^1,z^2)\in [0,T]\times \R^m\times\R\times\R\times\R^{2m}$,
\bes
G^{1n}(t,x,y^1,y^2,z^1,z^2)=\psi(\frac{y^1}{n},\frac{y^2}{n},\frac{z^1}{n},\frac{z^2}{n})\tilde{G}^{1n}(t,x,\psi_n(y^1),\psi_n(y^2),z^1,z^2),
\ees
\no where for each $n$, $\psi_n(y)$ is a continuous function for $y\in \R$, and $\psi_n(y)=1/n$ if $y\leq 0$; $\psi_n(y)=y$ if $y\geq 1/n$. We have the following properties:
\begin{equation}\label{abcd}
\left\{
\begin{array}[c]{l}%
\ (a)\  G^{1n}\  \text{is uniformly lipschitz w.r.t}\ (y^1,y^2,z^1,z^2);\smallskip\\
\ (b)\  |G^{1n}(t,x,y^1,y^2,z^1,z^2)|\leq C_fC_{\sigma}(1+|\varphi_n(x)|)|z^1|+C_h(1+|\varphi_n(x)|^{\gamma})(y^1)^+; \smallskip\\
\ (c)\ |G^{1n}(t,x,y^1,y^2,z^1,z^2)|\leq c_n, \ \text{for any}\  (t,x,y^1,y^2,z^1,z^2);\smallskip\\
\ (d)\ \text{For any}\  (t,x)\in [0,T]\times \R^m,  \varepsilon>0 \text{ and }\textbf{K}\ \text{a compact subset of}\  [\varepsilon, \frac{1}{\varepsilon}]^2\times \R^{2m},\\
\qquad \sup\limits_{(y^1,y^2,z^1,z^2)\in \textbf{K}}\abs{G^{1n}(t,x,y^1,y^2,z^1,z^2)-G^1(t,x,y^1,y^2,z^1,z^2)}\rightarrow 0,\ \text{as} \ n\rightarrow \infty.
\end{array}
\right.
\end{equation}

The same technique provides the sequence $(G^{2n})_{n\geq 1}$, which is indeed, the approximation of function $G^2$. For each $n\geq 1$ and $(t,x)\in [0,T]\times \R^m$, it is a direct result of \eqref{abcd}-(a) that (see \cite{peng1990}), there exist two pairs of processes $(\bar{Y}_s^{1n;(t,x)}, \bar{Z}_s^{1n;(t,x)})_{t\leq s\leq T}, (\bar{Y}_s^{2n;(t,x)}, \bar{Z}_s^{2n;(t,x)})_{t\leq s\leq T}\in \mathcal{S}_{t,T}^2(\R)\times \mathcal{H}_{t,T}^{2}(\R^m)$, which satisfy, for $s\in[t,T]$,
\begin{equation}\label{BSDE bar n}
\left\{
\begin{aligned}
\bar{Y}_s^{1n;(t,x)}&=e^{g^1(X_T^{t,x})}+\int_s^TG^{1n}(r,X_r^{t,x}, \bar{Y}_r^{1n;(t,x)},\bar{Y}_r^{2n;(t,x)},\bar{Z}_r^{1n;(t,x)},\bar{Z}_r^{2n;(t,x)})dr\\
&\qquad\qquad \qquad-\int_s^T\bar{Z}_r^{1n;(t,x)}dB_r;\smallskip\\
\bar{Y}_s^{2n;(t,x)}&=e^{g^2(X_T^{t,x})}+\int_s^TG^{2n}(r,X_r^{t,x}, \bar{Y}_r^{1n;(t,x)},\bar{Y}_r^{2n;(t,x)},\bar{Z}_r^{1n;(t,x)},\bar{Z}_r^{2n;(t,x)})dr\\
&\qquad\qquad \qquad-\int_s^T\bar{Z}_r^{2n;(t,x)}dB_r.
\end{aligned}
\right.
\end{equation}

\no Meanwhile, the properties \eqref{abcd}-(a),(c) and the result of El karoui et al. (ref. \cite{karoui}) yield that, there exist two sequences of deterministic measurable applications $\varsigma^{1n}(resp.\ \varsigma^{2n}): [0,T]\times \R^m \rightarrow \R$ and $\mathfrak{z}^{1n}(resp.\  \mathfrak{z}^{2n}): [0,T]\times \R^m\rightarrow \R^m$ such that for any $s\in [t,T]$,
\be \label{ybar}
\bar{Y}_s^{1n;(t,x)}= \varsigma^{1n}(s,X_s^{t,x})\quad (resp.\  \bar{Y}_s^{2n;(t,x)}=\varsigma^{2n}(s,X_s^{t,x}))
\ee 
\no and
\bes
\bar{Z}_s^{1n;(t,x)}= \mathfrak{z}^{1n}(s,X_s^{t,x}) \quad (resp.\  \bar{Z}_s^{2n;(t,x)}=\mathfrak{z}^{2n}(s,X_s^{t,x})).
\ees
\no Besides, we have the following deterministic expression: for $i=1,2,$ and $n\geq 1$,
\be\label{uin}
\varsigma^{in}(t,x)= \textbf{E}\Big[e^{g^i(X_T^{t,x})}+\int_t^TF^{in}(s,X_s^{t,x})ds\Big],\quad \forall(t,x)\in [0, T]\times \R^m,
\ee
\no where,
\bes
F^{in}(s,x)= G^{in}(s,x,\varsigma^{1n}(s,x),\varsigma^{2n}(s,x),\mathfrak{z}^{1n}(s,x),\mathfrak{z}^{2n}(s,x)).
\ees

\paragraph{\emph{Step 2.}} \label{step 2}\textit{Uniform integrability of $(\bar{Y}^{1n;(t,x)})_{n\geq 1}$ for fixed $(t,x)\in [0,T]\times \R^m$.}
\medskip

In this step, we will deal with the case of $i=1$, the case of $i=2$ can be treated in a similar way. For each $n\geq 1$, let us consider BSDE as follows, for $ s\in [t,T]$,
\begin{align}\label{BSDE tild}
\tilde{Y}_s^{1n}&=e^{g^1(X_T^{t,x})}+\int_s^T \left\{ C_fC_{\sigma}(1+|\varphi_n(X_r^{t,x})|)|\tilde{Z}_r^{1n}|+C_h(1+|\varphi_n(X_r^{t,x})|^{\gamma})(\tilde{Y}_r^{1n})^+\right\}dr \nonumber\\
&\quad -\int_s^T\tilde{Z}_r^{1n}dB_r.
\end{align}
\no Obviously, for any $x\in \R^m$ and integer $n\geq 1$,  the application which to $(y, z)\in \R^{1+m}$ associates $C_fC_{\sigma}(1+\varphi_n(x))|z|+C_h(1+\varphi_n(x))^{\gamma}y^+$ is Lipchitz continuous. Besides, $e^{g^1(X_T^{t,x})}\in \mathcal{L}^p(d\bp)$, $\forall p\geq 1$ which is the consequence of Assumption (A2)-(iii) and \eqref{estimate of X sigma exp}. Therefore, from the result of Pardoux and Peng \cite{peng1992}, we know that a pair of solutions $(\yts, \zts)_{t\leq  s\leq T}\in \mathcal{S}_{t,T}^p(\R)\times \mathcal{H}_{t,T}^p(\R^m)$ exists for any $p> 1$. Moreover through an adaptation of the result given by El Karoui et al. (1997,\cite{karoui}), we can infer the existence of deterministic measurable function $\tilde{\varsigma}^{1n}$: $[0,T]\times \R^m \rightarrow \R$ such that, for any $s\in[t,T]$,
\begin{equation}\label{estimubarn}
\tilde{Y}_s^{1n}= \tilde{\varsigma}^{1n}(s, X_s^{t,x}).
\end{equation}
\no Next let us consider the process 
$$
B^n_s= B_s- \int_0^s 1_{[t,T]}(r)C_fC_{\sigma}(1+|\varphi_n(X_r^{t,x})|)\mbox{sign}(\tilde{Z}_r^{1n})dr,\,\, 0\leq s\leq T,
$$
which is, thanks to Girsanov's Theorem, a Brownian motion under the probability $\bp^n$ on $(\Omega, \mathcal{F})$ whose density with respect to $\bp$ is
$$\zeta_T:=\zeta_T\{C_fC_{\sigma}(1+|\varphi_n(X_s^{t,x})|)\mbox{sign}(\tilde{Z}_s^{1n})1_{[t,T]}(s)\},$$
where for any $z=(z^i)_{i=1,...,d}\in \R^m$, $\mbox{sign}(z)=(1_{[|z^i|\neq 0]}\frac{z^i}{|z^i|})_{i=1,...,d}$ and  $\zeta_T(\cdot)$ is
defined by (\ref{density fun}). Then (\ref{BSDE tild}) becomes
\begin{equation*}
\tilde{Y}_s^{1n}= e^{g^1(X_T^{t,x})}+\int_s^T
C_h(1+|\varphi_n(X_r^{t,x})|^{\gamma})(\tilde{Y}_r^{1n})^+dr-
\int_s^T \tilde{Z}_r^{1n} dB^n_r, \quad t\leq s\leq T.
\end{equation*}
\noindent Therefore, taking into account of (\ref{estimubarn}), we deduce,
\begin{equation*}
\tilde{\varsigma}^{1n}(t,x)= \textbf{E}^n\Big[e^{g^1(X_T^{t,x})+\int_t^T
C_h\left(1+\abs{\varphi_n(X_s^{t,x})}^{\gamma}\right)ds}|\mathcal{F}_t\Big],
\end{equation*}
\noindent where $\textbf{E}^n$ is the expectation under probability $\bp^n$.
Taking the expectation on both sides under the probability $\bp^n$
and considering $\tilde{\varsigma}^{1n}(t,x)$ is deterministic, one obtains,
\begin{equation*}
\tilde{\varsigma}^{1n}(t,x)= \textbf{E}^n\Big[e^{g^1(X_T^{t,x})+\int_t^T
C_h\left(1+\abs{\varphi_n(X_s^{t,x})}^{\gamma}\right)ds}\Big].
\end{equation*}
\noindent Then by the Assumption (A2)-(iii) we have: $\forall (t,x)\in \esp$, 
\begin{equation*}
 \begin{aligned}
\abs{\tilde{\varsigma}^{1n}(t,x)}&\leq \textbf{E}^n\Big[e^{C\sup_{0\leq s\leq T}(1+ \abs{X_s^{t,x}}^{\gamma})}\Big]\\
        &= \textbf{E}\Big[e^{C\sup_{0\leq s\leq T}(1+ \abs{X_s^{t,x}}^{\gamma})}\cdot\zeta_T\Big]. 
 \end{aligned}
\end{equation*}
\no By Lemma \ref{density function lp bounded}, there exists some $1<p_0<2$ (which does not depend on $(t,x)$), such that $\textbf{E}[|\zeta_T|^{p_0}]<\infty$. Applying Young's inequality, besides, considering \eqref{estimate of X sigma exp} yield that, 
\begin{equation*}
\begin{aligned}
|\tilde{\varsigma}^{1n}(t,x)|&\leq \textbf{E}\big[e^{\frac{Cp_0}{p_0-1}\sup_{0\leq s\leq T}(1+ \abs{X_s^{t,x}}^{\gamma})}\big]+\textbf{E}\big[|\zeta_T|^{p_0}\big]\\
&\leq e^{C(1+|x|^{\gamma})}.
\end{aligned}
\end{equation*}
\no Next taking into account point \eqref{abcd}-(b) and using comparison Theorem of BSDEs,  we obtain for any $s\in [t,T]$,
$$\tilde{Y}^{1n}_s=\tilde{\varsigma}^{1n}(s, X^{t,x}_s)\ge 
\bar{Y}^{1n;(t,x)}_s=\varsigma^{1n}(s, X^{t,x}_s).
$$
Then, by choosing $s=t$, we get that $\varsigma^{1n}(t,x)\leq e^{C(1+|x|^{\gamma})}$, $(t,x)\in \esp$. But in a similar way one can show that for any $(t,x)\in \esp$, $\varsigma^{1n}(t,x)\geq e^{-C(1+|x|^{\gamma})}$. Therefore,
\begin{equation}\label{growthu}
 e^{-C(1+|x|^{\gamma})}\leq \varsigma^{1n}(t,x)\leq e^{C(1+|x|^{\gamma})},\ (t,x)\in \esp.
\end{equation}
\no By \eqref{growthu},\eqref{ybar} and \eqref{estimate of X sigma exp}, we conclude, 
$\bar{Y}_s^{1n;(t,x)}\in \mathcal{S}_{t,T}^p(\R^m)$ holds, i.e., for any $p>1$, we have,
\begin{equation}\label{ybn belongs to sp}
\textbf{E}\Big[\sup_{t\leq s\leq T}\big|\bar{Y}_s^{1n;(t,x)}\big|^p\Big]<\infty.
\end{equation}
%--------------------------------------------------------------------------------
\paragraph{\emph{Step 3.}}\label{step3} \textit{Uniform integrability of $(\bar{Z}_s^{1n;(t,x)})_{t\leq s\leq T}$.}
\medskip

\noindent Recalling the equation \eqref{BSDE bar n} and making use of It\^o's formula with $(\bar{Y}_s^{1n;(t,x)})^2$, we obtain, in a standard way, the following result. 
\medskip

There exists a constant $C$ independent of $n$ and $t,x$ such that for any $t\leq T$, for $i=1,2$,
\begin{equation}\label{estimate of zbn}
\textbf{E}\big[\int_t^T|\bar{Z}_s^{1n;(t,x)}|^2ds\big]\le C.
\end{equation}
The proof is omitted for conciseness.
%%%%%%%%%%%%%%%%%%%%%%%%%%%%%%%%%%%%%%%%
\paragraph{\emph{Step 4.}} \textit{There exists a subsequence of $((\bar{Y}_s^{1n;(0,x)},\bar{Z}_s^{1n;(0,x)})_{0\leq s\leq T})_{n\geq 1}$ which converges respectively to $(\bar{Y}^1_s, \bar{Z}^1_s)_{0\leq s\leq T}$, solution of the BSDE \eqref{BSDE exp}}. Moreover, $\bar{Y}_s^1>0$, $\forall s\in [0,T]$, $\bp$-a.s.
\medskip

Let us recall the expression (\ref{uin}) for case $i=1$,
\be\label{uin re}
\varsigma^{1n}(t,x)=\textbf{E}\Big[e^{g^1(\xT)}+\int_t^TF^{1n}(s,\xs)ds\Big],\qquad \forall (t,x)\in [0,T]\times \R^m.
\ee
\no We now apply property \eqref{abcd}-(b) in Step 1 combined with the uniform estimates (\ref{ybn belongs to sp}), (\ref{estimate of zbn}) and the Young's inequality to show that, for $1<q<2$,
\begin{equation}\label{F1n lp}
\begin{aligned}
\textbf{E}\Big[\int_{0}^T&|F^{1n}(s,\xso)|^qds\Big]\\
&\leq C\textbf{E}\Big[\int_{0}^T(1+|\varphi_n(\xso)|)^q|\bar{Z}_s^{1n;(0,x)}|^q+(1+|\varphi_n(\xso)|)^{\gamma q}|\bar{Y}_s^{1n;(0,x)}|^qds\Big]\\
&\leq  C\textbf{E}\Big[\Big(\int_{0}^T |\bar{Z}_s^{1n;(0,x)}|^2ds\Big)^{\frac{q}{2}}\Big(\int_{0}^T(1+|\xso|)^{\frac{2q}{2-q}}ds\Big)^{\frac{2-q}{2}}\Big]\\
&\qquad + C\textbf{E}\Big[ \sup_{0\leq s\leq T}{|\bar{Y}_s^{1n;(0,x)} |^q}\cdot\int_{0}^T (1+|\xso|)^{\gamma q}ds\Big]\\
&\leq C\{\textbf{E}\Big[\int_{0}^T |\bar{Z}_s^{1n;(0,x)}|^2ds\Big]+ \textbf{E}\Big[ \sup_{0\leq s\leq T}|\bar{Y}_s^{1n;(0,x)} |^2\Big]+1\}\\
&<\infty.
\end{aligned}
\end{equation}
\no Therefore, there exists a sub-sequence $\{n_k\}$ (for notation simplification, we still denote it by $\{n\}$) and a $\mathcal{B}([0, T])\otimes \mathcal{B}(\R^m)$-measurable deterministic function $F^1(s,y)$ such that:
\be\label{Fn weakly cov}
F^{1n}\rightarrow F^1 \ \text{weakly in}\ \mathcal{L}^q([0,T]\times \R^m; \mu(0,x;s,dy)ds).
\ee
Next we aim to prove that $(\varsigma^{1n}(t,x))_{n\geq 1}$ is a Cauchy
sequence for each $(t,x)\in[0,T]\times \R^m$. Now let $(t,x)$ be fixed, $\eta>0$, $k$, $n$ and $m\geq1$ be integers. From (\ref{uin re}), we have,
\begin{equation*}
\begin{aligned}
\abs{\varsigma^{1n}(t,x)-\varsigma^{1m}(t,x)}&=
\Big|\textbf{E}\Big[\int_t^TF^{1n}(s,X_s^{t,x})-F^{1m}(s,X_s^{t,x})ds\Big]\Big|\\
&\leq
\textbf{E}\Big[\int_t^{t+\eta}\left|F^{1n}(s,X_s^{t,x})-F^{1m}(s,X_s^{t,x})\right|ds\Big]\\
&\quad +\Big|\textbf{E}\Big[\int_{t+\eta}^T\left(F^{1n}(s,X_s^{t,x})-F^{1m}(s,X_s^{t,x})\right).\mathbbm{1}_{\{\abs{X_s^{t,x}}\leq
k\}}ds\Big]\Big|\\
&\quad +\Big|\textbf{E}\Big[\int_{t+\eta}^T\left(F^{1n}(s,X_s^{t,x})-F^{1m}(s,X_s^{t,x})\right).\mathbbm{1}_{\{\abs{X_s^{t,x}}>
k\}}ds\Big]\Big|,
\end{aligned}
\end{equation*}
\no where on the right side, noticing (\ref{F1n lp}), we obtain,
\begin{equation*}
\begin{aligned}
\textbf{E}&\Big[\int_t^{t+\eta}\abs{F^{1n}(s, X_s^{t,x})- F^{1m}(s,
X_s^{t,x})}ds\Big]\\
& \leq
\eta^{\frac{q-1}{q}}\{\textbf{E}\Big[\int_0^T\abs{F^{1n}(s,X_s^{t,x})-F^{1m}(s,X_s^{t,x})}^qds\Big]\}^{\frac{1}{q}}\leq C\eta^{\frac{q-1}{q}}.
\end{aligned}
\end{equation*}

\noindent At the same time, Corollary \ref{dominated cor} associates with the $\mathcal{L}^{\frac{q}{q-1}}$-domination property implies:
\begin{equation*}
\begin{aligned}
\Big|\textbf{E}\Big[\int_{t+\eta}^T&\left(F^{1n}(s,X_s^{t,x})-F^{1m}(s,X_s^{t,x})\right).\mathbbm{1}_{\{\abs{X_s^{t,x}}\leq
k\}}ds\Big]\Big|\\
&=
\Big|\int_{\R^m}\int_{t+\eta}^T(F^{1n}(s,\eta)-F^{1m}(s,\eta)).\mathbbm{1}_{\{\abs{\eta}\leq
k\}}\mu(t,x;s,d\eta)ds\Big|\\
&=
\Big|\int_{\R^m}\int_{t+\eta}^T(F^{1n}(s,\eta)-F^{1m}(s,\eta)).\mathbbm{1}_{\{\abs{\eta}\leq
k\}}\phi_{t,x}(s,\eta)\mu(0,x;s,d\eta)ds\Big|.
%&\leq
%\int_{\R^m}\left(\int_{t+\eta}^T\left(F^{1n}(s,\eta)-F^{1m}(s,\eta)\right)^qds\right)^{\frac{1}{q}}\left(\int_{t+\eta}^T\left(\mathbbm{1}_{\{\abs{\eta}\leq
%k\}}\phi_{t,x}(s,\eta)\right)^{\frac{q}{q-1}}\mu(0,a;s,d\eta)ds\right)^{\frac{q-1}{q}}.
\end{aligned}
\end{equation*}

\noindent Since $\phi_{t,x}(s,\eta) \in \mathcal{L}^{\frac{q}{q-1}}([t+\eta,
T]\times [-k, k]^m$; $\mu(0,x; s, d\eta)ds)$, for $k\geq 1$, it
follows from (\ref{Fn weakly cov}) that for each $(t,x)\in [0, T]\times \R^m$, we have,
\begin{equation*}
\textbf{E}\Big[\int_{t+\eta}^T\left(F^{1n}(s, X_s^{t,x})-F^{1m}(s,
X_s^{t,x})\right)\mathbbm{1}_{\{\abs{X_s^{t,x}}\leq
k\}}ds\Big]\rightarrow 0\text{ as } n,m\rightarrow \infty.
\end{equation*}
\noindent Finally,
\begin{equation*}
\begin{aligned}
\Big|\textbf{E}\Big[\int_{t+\eta}^T&\left(F^{1n}(s,X_s^{t,x})-F^{1m}(s,X_s^{t,x})\right).\mathbbm{1}_{\{\abs{X_s^{t,x}}>
k\}}ds\Big]\Big|\\
&\leq C\{\textbf{E}\Big[\int_{t+\eta}^T
\mathbbm{1}_{\left\{\abs{X_s^{t,x}}>k\right\}}ds\Big]\}^{\frac{q-1}{q}}\{\textbf{E}\Big[\int_{t+\eta}^T\abs{F^{1n}(s,X_s^{t,x})-F^{1m}(s,X_s^{t,x})}^qds\Big]\}^{\frac{1}{q}}\\
&\leq Ck^{-\frac{q-1}{q}}
\end{aligned}
\end{equation*}

Therefore, for each $(t,x)\in [0,T]\times \R^m$, $(\varsigma^{1n}(t,x))_{n\geq 1}$ is a Cauthy sequence and then there
exists a borelian application $\varsigma^1$ on $[0,T]\times \R^m$, such that for each $(t,x)\in [0,T]\times \R^m$, $\lim_{n\rightarrow \infty}\varsigma^{1n}(t,x)=\varsigma^1(t,x)$, which indicates that for $t\in [0,T]$, $\lim_{n\rightarrow \infty} \bar{Y}_t^{1n;(0,x)}(\omega)=\varsigma^1(t,\xto),\quad \bp-a.s.$ Taking account of (\ref{ybn belongs to sp}) and the Lebesgue dominated convergence theorem, we obtain the sequence $((\bar{Y}_t^{1n;(0,x)})_{0\leq t\leq T})_{n\geq 1}$ converges to $\bar{Y}^1=(\varsigma^1(t, \xto))_{0\leq t\leq T}$ in $\mathcal{L}^p([0,T]\times\R^m)$ for any $p>1$, that is:
\be \label{conv ytn}
\textbf{E}\Big[\int_{0}^T|\bar{Y}_t^{1n;(0,x)}-\bar{Y}_t^1|^pdt\Big]\rightarrow 0, \quad \text{as}\ n\rightarrow \infty.
\ee
\medskip

Next, we will show that for any $p>1$, $\bar{Z}^{1n;(0,x)}=(\mathfrak{z}^{1n}(t,\xto))_{0\leq t\leq T})_{n\geq 1}$ has a limit in $\mathcal{H}_{T}^2(\R^m)$. Besides, $(\bar{Y}^{1n;(0,x)})_{n\geq 1}$ is convergent in $\mathcal{S}_{T}^2(\R)$ as well.
\medskip

\noindent We now focus on the first claim. For $n,m \geq 1$ and $0 \leq t\leq T$, using It\^{o}'s formula with {$(\ybtn-\ybtm)^2$} (we omit the subscript $(0,x)$ for convenience) and considering \eqref{abcd}-(b) in Step 1, we get,
\begin{align*}
&\big|\ybtn-\ybtm\big|^2+ \int_t^T \abs{\zbsn-\zbsm}^2ds\\
&= 2\int_t^T(\ybsn-\ybsm)(G^{1n}(s,
X_s^{0,x},\ybsn,\bar{Y}_s^{2n},\zbsn,\bar{Z}_s^{2n})-\\
&\quad -G^{1m}(s,
X_s^{0,x},\ybsm,\bar{Y}_s^{2m},\zbsm,\bar{Z}_s^{2m}))ds-2\int_t^T(\ybsn-\ybsm)(\zbsn-\zbsm)dB_s\\
&\leq C\int_t^T
\abs{\ybsn-\ybsm}\Big[(\abs{\zbsn}+\abs{\zbsm})(1+|X_s^{0,x}|)+(|\ybsn|+|\ybsm|)(1+|X_s^{0,x}|)^{\gamma}\Big]ds\\
&\quad - 2\int_t^T(\ybsn-\ybsm)(\zbsn-\zbsm)dB_s.
\end{align*}
\no Since for any $x,y,z\in \R$, $|xyz|\leq \frac{1}{a}|x|^a+\frac{1}{b}|y|^b+\frac{1}{c}|z|^c$ with $\frac{1}{a}+\frac{1}{b}+\frac{1}{c}=1$, then, for any $\varepsilon>0$, we have,
\begin{equation}\label{4.27}
\begin{aligned}
&|\bar{Y}_t^{1n}-\bar{Y}_t^{1m}|^2+\int_t^T|\bar{Z}_s^{1n}-\bar{Z}_s^{1m}|^2ds\\
&\leq C\Big\{\frac{\varepsilon^2}{2}\int_t^T(|\bar{Z}_s^{1n}|+|\bar{Z}_s^{1m}|)^2ds+\frac{\varepsilon^4}{4}\int_t^T(1+|X_s^{0,x}|)^4ds+\frac{1}{4\varepsilon^8}\int_t^T|\bar{Y}_s^{1n}-\bar{Y}_s^{1m}|^4ds\\
&\quad+\frac{\varepsilon^2}{2}\int_t^T(|\bar{Y}_s^{1n}|+|\bar{Y}_s^{1m}|)^2ds+\frac{\varepsilon^4}{4}\int_t^T(1+|X_s^{0,x}|)^{4\gamma} ds+\frac{1}{4\varepsilon^8}\int_t^T|\bar{Y}_s^{1n}-\bar{Y}_s^{1m}|^4ds\Big\}\\
&\quad-2\int_t^T(\bar{Y}_s^{1n}-\bar{Y}_s^{1m})(\bar{Z}_s^{1n}-\bar{Z}_s^{1m})dB_s.
\end{aligned}
\end{equation}
\no Taking now $t=0$ in \eqref{4.27}, expectation on both sides and the limit w.r.t. $n$ and $m$, we deduce that, 
\begin{equation}\label{4.29}
\limsup_{n,m\rightarrow \infty} \textbf{E}\Big[\int_{0}^T|\bar{Z}_s^{1n}-\bar{Z}_s^{1m}|^2ds\Big]\leq C\{\frac{\varepsilon^2}{2}+\frac{\varepsilon^4}{4}\},
\end{equation}
due to \eqref{estimate of zbn}, \eqref{estimate of X sigma} and the convergence of \eqref{conv ytn}. As $\varepsilon$ is arbitrary, then the sequence $(\bar{Z}^{1n})_{n\geq 1}$ is convergent in $\mathcal{H}_{T}^2(\R^m)$ to a process $Z^1$.
\medskip

Now, returning to inequality \eqref{4.27}, taking the supremum over $[0,T]$ and using BDG's inequality, we obtain that, 
\begin{align*}
&\textbf{E}\Big[\sup_{0\leq t\leq T}|\bar{Y}_t^{1n}-\bar{Y}_t^{1m}|^2+\int_{0}^T|\bar{Z}_s^{1n}-\bar{Z}_s^{1m}|^2ds\Big]\\
&\leq C\{\frac{\varepsilon^2}{2}+\frac{\varepsilon^4}{4}\}+\frac{1}{4}\textbf{E}\Big[\sup_{0\leq t\leq T}|\bar{Y}_t^{1n}-\bar{Y}_t^{1m}|^2\Big]+4\textbf{E}\Big[\int_{0}^T|\bar{Z}_s^{1n}-\bar{Z}_s^{1m}|^2ds\Big]
\end{align*}
which implies that
$$
\limsup_{n,m\rightarrow \infty}\textbf{E}\Big[\sup_{0\leq t\leq T}|\bar{Y}_t^{1n}-\bar{Y}_t^{1m}|^2\Big]=0,
$$
since $\varepsilon$ is arbitrary and \eqref{4.29}. Thus, the sequence of $(\bar{Y}^{1n})_{n\geq 1}$ converges to $\bar{Y}^1$ in $\mathcal{S}_{T}^2(\R)$ which is a continuous process.
\medskip

Next, note that since $\varsigma^{1n}(s,x)\geq e^{-C(1+|x|^\gamma)}$, then, $\bar{Y}_s^1>0$, $\forall s\leq T$, $\bp$-a.s.
\medskip

Finally, repeat the procedure for player $i=2$, we have also the convergence of $(\bar{Z}^{2n})_{n\geq 1}$ (resp. $(\bar{Y}^{2n})_{n\geq 1}$) in $\mathcal{H}_{T}^2(\R^m)$ (resp. $\mathcal{S}_{T}^2(\R)$) to $\bar{Z}^2$ (resp. $\bar{Y}^2=\varsigma^2(.,  X^{0,x}.)$).

\paragraph{\emph{Step 5.}}\label{step 5}\textit{The limit process $(\bar{Y}_t^i,\bar{Z}_t^i)_{0\leq t\leq T}$ (i=1,2) is the solution of BSDE (\ref{BSDE exp}). }

Indeed, we need to show that (for case $i=1$):
\bes
F^1(t,\xto)=G^1(t,\xto,\ybt,\ybtt,\zbt,\zbtt)\quad dt\otimes d\bp- a.s.
\ees
\no For $k\geq 1$, we have, 
\begin{align} \label{G}
&\textbf{E}\bigg[\int_{0}^T |G^{1n}(s, \xso, \ybsn, \ybsnt,\zbsn,\zbsnt)- G^1(s,\xso, \ybs, \ybst,\zbs,\zbst)|ds\bigg]\nonumber\\
\leq  \ &\textbf{E}\bigg[\int_{0}^T |G^{1n}(s, \xso, \ybsn, \ybsnt,\zbsn,\zbsnt) - G^1(s,\xso, \ybsn, \ybsnt,\zbsn,\zbsnt)|\cdot \nonumber\\ 
&\qquad \cdot\mathbbm{1}_{\{\frac{1}{k}< |\ybsn|+|\ybsnt|+|\zbsn|+|\zbsnt|<k\}}ds\bigg]\nonumber\\
&+\textbf{E}\bigg[\int_{0}^T |G^{1n}(s, \xso, \ybsn, \ybsnt,\zbsn,\zbsnt)- G^1(s,\xso, \ybsn, \ybsnt,\zbsn,\zbsnt)|\cdot\nonumber\\
&\qquad \cdot \mathbbm{1}_{\{|\ybsn|+|\ybsnt|+|\zbsn|+|\zbsnt|\geq k\}}ds\bigg]\nonumber\\
&+\textbf{E}\bigg[\int_{0}^T |G^{1n}(s, \xso, \ybsn, \ybsnt,\zbsn,\zbsnt)- G^1(s,\xso, \ybsn, \ybsnt,\zbsn,\zbsnt)|\cdot\nonumber\\
&\qquad \cdot \mathbbm{1}_{\{|\ybsn|+|\ybsnt|+|\zbsn|+|\zbsnt|\leq \frac{1}{k}\}}ds\bigg]\nonumber\\
&+\textbf{E}\bigg[\int_{0}^T |G^1(s, \xso, \ybsn, \ybsnt,\zbsn,\zbsnt)- G^1(s,\xso, \ybs, \ybst,\zbs,\zbst)|ds\bigg]\nonumber\\
&:= I_1^n+I_2^n+I_3^n+I_4^n,
\end{align}
\no where the sequence $I_1^n$, $n\geq 1$ converges to 0. On one hand, for $n\geq 1$, the point \eqref{abcd}-(b) in Step 1 implies that,
\beas
&\left|G^{1n}(s, \xso, \ybsn, \ybsnt,\zbsn,\zbsnt)- G^1(s,\xso, \ybsn, \ybsnt,\zbsn,\zbsnt)\right|\cdot\\
&\qquad \qquad\cdot\mathbbm{1}_{\{\frac{1}{k}<|\ybsn|+|\ybsnt|+|\zbsn|+|\zbsnt|<k\}}\\
& <C_fC_{\sigma}(1+|\xso|)k+C_h(1+|\xso|^{\gamma})k.
\eeas
\no On the other hand, considering the point \eqref{abcd}-(d), we obtain that, 
\beas
&\left|G^{1n}(s, \xso, \ybsn, \ybsnt,\zbsn,\zbsnt)- G^1(s,\xso, \ybsn, \ybsnt,\zbsn,\zbsnt)\right|\cdot \\
&\qquad\qquad \cdot\mathbbm{1}_{\{\frac{1}{k}<|\ybsn|+|\ybsnt|+|\zbsn|+|\zbsnt|<k\}}\\
&\leq\quad \sup\limits_{\substack{(y_s^1,y_s^2,z_s^1,z_s^2)\\\frac{1}{k}<|y_s^1|+|y_s^2|+|z_s^1|+|z_s^1|<k}}\left|G^{1n}(s,\xso,y_s^1,y_s^2,z_s^1,z_s^2)-G^1(x,\xso,y_s^1,y_s^2,z_s^1,z_s^2)\right|\rightarrow 0,
\eeas
\no as $n\rightarrow \infty$. Thanks to Lebesgue's dominated convergence theorem, the sequence $I_1^n$ of (\ref{G}) converges to 0 in $\mathcal{H}_{T}^1(\R)$.
\medskip

\no The sequence $I_2^n$ in \eqref{G} is bounded by $\frac{C}{k^{2(q-1)/q}}$ with $q\in (1,2)$. Actually, from point \eqref{abcd}-(b) and Markov's inequality, for $q\in (1,2)$, we get,
\begin{align*}
I_2^n &\leq C\Big\{\textbf{E}\Big[\int_{0}^T(1+|\xso|)^q|\zbsn|^q+(1+|\xso|^{\gamma})^q|\ybsn|^qds\Big]\Big\}^{\frac{1}{q}}\times\\
&\qquad\qquad\qquad\times \Big\{\textbf{E}\Big[\int_{0}^T \mathbbm{1}_{\{|\ybsn|+|\ybsnt|+|\zbsn|+|\zbsnt|\geq k\}}ds\Big]\Big\}^{\frac{q-1}{q}}\\
%&\leq \times \\
&\leq  C\Big\{\textbf{E}\Big[\int_{0}^T|\zbsn|^2ds\Big]+\textbf{E}\Big[\int_{0}^T (1+|\xso|)^{\frac{2q}{2-q}}ds\Big]\\
&\qquad\qquad\qquad+\textbf{E}\Big[\int_{0}^T|\ybsn|^2ds\Big]+\textbf{E}\Big[\int_{0}^T(1+|\xso|)^{\gamma\cdot\frac{2q}{2-q}}ds\Big]\Big\}^{\frac{1}{q}}\times \\
&\quad\times \frac{\Big\{\textbf{E}\Big[\int_{0}^T|\ybsn|^2+|\ybsnt|^2+|\zbsn|^2+|\zbsnt|^2ds\Big]\Big\}^{\frac{q-1}{q}}}{(k^2)^{\frac{q-1}{q}}}\\
&\leq \frac{C}{k^{\frac{2(q-1)}{q}}}.
\end{align*}
\no The last inequality is a straightforward result of the estimates (\ref{estimate of theta})(\ref{ybn belongs to sp}) and (\ref{estimate of zbn}). 
\medskip

The third sequence $I_3^n$ in \eqref{G} is bounded by $C/k$ with constant $C$ independent on $k$. Indeed, by \eqref{abcd}-(b) and \eqref{estimate of X sigma},
$$
I_3^n\leq \textbf{E}\Big[\int_{0}^TC_fC_{\sigma}(1+|X_s^{0,x}|)\frac{1}{k}+C_h(1+|X_s^{0,x}|^{\gamma})\frac{1}{k}ds\Big]\leq C/k.
$$

The fourth sequence $I_4^n$, $n\geq 1$ in \eqref{G} also converges to $0$, at least for a subsequence. Actually, since the sequence $(\bar{Z}^{1n})_{n\geq 1}$ converges to $\bar{Z}^1$ in $\mathcal{H}_{T}^2(\R^m)$, then there exists a subsequence $(\bar{Z}^{1n_k})_{k\geq 1}$ such that it converges to $\bar{Z}^1$, $dt\otimes d\bp$-a.e., and furthermore, $\sup_{k\geq 1}|\bar{Z}_t^{1n_k}(\omega)|\in \mathcal{H}_{T}^2(\R)$. On the other hand, $(\bar{Y}^{1n_k})_{k\geq 1}$ converges to $\bar{Y}^1>0$, $dt\otimes d\bp$-a.e.. Thus, taking the continuity of function $G^1(t,x,y^1,y^2,z^1,z^2)$ w.r.t $(y^1,y^2,z^1,z^2)$ into account, we obtain that
\beas
G^1(t,\xto,\bar{Y}_t^{1n_k},\bar{Y}_t^{2n_k},\bar{Z}_t^{1n_k},\bar{Z}_t^{2n_k})\rightarrow_{k\rightarrow\infty} G^1(t,\xto,\ybt,\ybtt,\zbt,\zbtt)\quad \ dt\otimes d\bp-a.e.
\eeas
In addition, considering that
\bes
\sup_{k\geq 1}|G^1(t,\xto,\bar{Y}_t^{1n_k},\bar{Y}_t^{2n_k},\bar{Z}_t^{1n_k},\bar{Z}_t^{1n_k})| \in \mathcal{H}_{T}^q(\R)\text{ for } 1<q<2,
\ees
which follows from (\ref{F1n lp}). Finally, by the dominated convergence theorem, one can get that,
\beas
G^1(t,\xto,\bar{Y}_t^{1n_k},\bar{Y}_t^{2n_k},\bar{Z}_t^{1n_k},\bar{Z}_t^{2n_k})\rightarrow_{k\rightarrow \infty} G^1(t,\xto,\ybt,\ybtt,\zbt,\zbtt)\quad \text{in}\ \mathcal{H}_{T}^q(\R),
\eeas
which yields to the convergence of $I_4^n$ in (\ref{G}) to $0$.
\medskip

\no It follows that the sequence $(G^{1n}(t,\xto,\ybtn,\ybtnt,\zbtn,\zbtnt)_{0\leq t\leq T})_{n\geq 1}$ converges to \\
 $(G^1(t,\xto,\ybt,\ybtt,\zbt,\zbtt))_{0\leq t\leq T}$ in $\mathcal{L}^1([0,T]\times \Omega, dt\otimes d\bp)$ and then \\$F^1(t,\xto)=G^1(t,\xto,\ybt,\ybtt,\zbt,\zbtt), dt\otimes d\bp$-a.e. In the same way, we have, $F^2(t, X_t^{0,x})=G^2(t, X_t^{0,x}, \bar{Y}_t^1, \bar{Y}_t^2, \bar{Z}_t^1, \bar{Z}_t^2)$, $dt\otimes d\bp$-a.e. Thus, the processes $(Y^i, Z^i)$, $i=1,2$ are the solutions of the backward equation (\ref{BSDE exp}).

\paragraph{\emph{Step 6.}}\label{Step 6} \textit{The solutions $(Y_t^i,Z_t^i),\ i=1,2$ for BSDE \eqref{main BSDE in th} exist.}
\medskip

Obviously observed from (\ref{growthu}) that $\ybt$ is strict positive which enable us to obtain the solution of the original BSDE (\ref{main BSDE in th}) by:
\bes
\left\{
\begin{aligned}
Y_t^1&=\ln\ybt;\\
Z_t^1&=\frac{\zbt}{\ybt},\quad t\in [0,T].
\end{aligned}
\right.
\ees 
\no The same illustrate about the case $i=2$ gives the existence of solution $(Y^2,Z^2)$ for BSDE \eqref{main BSDE in th}. Besides, it follows from the fact $\bar{Y}_t^i=\varsigma^i(t, X_t^{0,x})$ and for each $(t,x)\in [0,T]\times \R^m$, $e^{-C(1+|x|^{\gamma})}\leq \varsigma^i(t,x)\leq e^{C(1+|x|^{\gamma})}$ with $1<\gamma<2$, that $Y^i$ also has a representation through a deterministic function $\varpi^i(t,x)=\ln \varsigma^i(t,x)$ which is of subquadratic growth, i.e. $|\varpi^i(t,x)|\leq C(1+|x|^{\gamma})$ with $1<\gamma<2$, $i=1,2$. The proof is completed.

\end{proof}

\end{document}